\documentclass[reqno,11pt]{amsart}
\usepackage{amssymb}
\usepackage{amsfonts}
\usepackage{graphicx}
\usepackage{wrapfig}
\usepackage{float}
\usepackage{subfig}
\usepackage{cite}
\usepackage[top=1.25in,bottom=1.25in,left=1.25in,right=1.25in]{geometry}

\newtheorem{theorem}{Theorem}
\theoremstyle{plain}

\newtheorem{corollary}{Corollary}[section]

\newtheorem{definition}{Definition}[section]
\newtheorem{example}{Example}[section]

\newtheorem{lemma}{Lemma}[section]

\newtheorem{problem}{Problem}[section]
\newtheorem{proposition}{Proposition}[section]
\newtheorem{remark}{Remark}[section]

\numberwithin{equation}{section}

\begin{document}
\title{On the Ramified Optimal Allocation Problem }
\author{Qinglan Xia}
\address{University of California at Davis\\
Department of Mathematics\\
Davis, CA, 95616}
\email{qlxia@math.ucdavis.edu}
\urladdr{http://math.ucdavis.edu/\symbol{126}qlxia}
\author{Shaofeng Xu}
\address{University of California at Davis\\
Department of Economics\\
Davis, CA, 95616}
\email{sxu@ucdavis.edu}
\subjclass[2000]{Primary 91B32, 58E17; Secondary 49Q20, 90B18. \textit{%
Journal of Economic Literature Classification.} D61, C60, R12, R40.}
\keywords{ramified transportation, transport economy of scale, optimal
transport path, branching structure, allocation problem, assignment map,
state matrix}
\thanks{This work is supported by an NSF grant DMS-0710714.}

\begin{abstract}
This paper proposes an optimal allocation problem with ramified transport
technology in a spatial economy. Ramified transportation is used to model
the transport economy of scale in group transportation observed widely in
both nature and efficiently designed transport systems of branching
structures. The ramified allocation problem aims at finding an optimal
allocation plan as well as an associated optimal allocation path to minimize
overall cost of transporting commodity from factories to households. This
problem differentiates itself from existing ramified transportation
literature in that the distribution of production among factories is not
fixed but endogenously determined as observed in many allocation practices.
It's shown that due to the transport economy of scale in ramified
transportation, each optimal allocation plan corresponds equivalently to an
optimal assignment map from households to factories. This optimal assignment
map provides a natural partition of both households and allocation paths. We
develop methods of marginal transportation analysis and projectional
analysis to study properties of optimal assignment maps. These properties
are then related to the search for an optimal assignment map in the context
of state matrix.
\end{abstract}

\maketitle

\section{Introduction}

One of the lasting interests in economics is to study optimal resource
allocation in a spatial economy. For instance, the well known
Monge-Kantorovich transport problem aims at finding an efficient allocation
plan or map for transporting some commodity from factories to households.
This problem was pioneered by Monge \cite{monge} and advanced fully by
Kantorovich \cite{kantorovich} who won the Nobel prize in economics in 1975
for his seminal work on optimal allocation of resources. Recent advancement
of this problem in mathematics can be found in Villani \cite%
{villani1,villani} and references therein. Monge-Kantorovich problem has
also been applied to study other related economic problems, e.g., spatial
firm pricing (Buttazzo and Carlier \cite{Buttazzo}), principal-agent problem
(Figalli, Kim and McCann \cite{Figalli}), hedonic equilibrium models
(Chiappori, McCann and Nesheim \cite{Chiappori1}; Ekeland \cite{Ekeland1}),
matching and partition in labor market (Carlier and Ekeland \cite{Carlier};
McCann and Trokhimtchouk \cite{McCan}).

Recently, a new research field known as \textit{Ramified Optimal
Transportation }has grown out of Monge-Kantorovich problem. Representative
studies can be found for instance in Gilbert \cite{gilbert}, Xia \cite%
{xia1,xia2,xia6,xia8,xia3,xia11,xia12}, Maddalena, Solimini and Morel \cite%
{msm}, Bernot, Caselles and Morel \cite{BCM,book}, Brancolini, Buttazzo and
Santambrogio \cite{buttazzo0}, Santambrogio \cite{Santa}, Devillanova and
Solimini \cite{Solimini}, Xia and Vershynina \cite{xia9}. Ramified optimal
transport problem studies how to find an \textit{optimal transport path}
from sources to targets as shown in Figure \ref{figure1}. Different from the
standard Monge-Kantorovich problem where the transportation cost is solely
determined by a transport plan or map, the transportation cost in the
ramified transport problem is determined by the actual transport path which
transports the commodity from sources to targets. Ramified transportation
indeed formally formulates the concept of \textit{transport economy of scale}
in group transportation observed widely in both nature (e.g. trees, blood
vessels, river channel networks, lightning) and efficiently designed
transport systems of branching structures (e.g. railway configurations and
postage delivery networks). An application of ramified optimal
transportation in economics can be found in Xia and Xu \cite{xia13}, which
showed that a well designed ramified transport system can improve the
welfare of consumers in the system.
\begin{figure}[tbp]
\includegraphics[width=12cm]{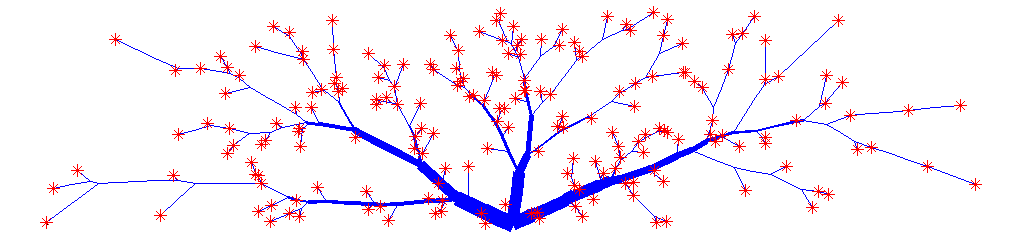}
\caption{An optimal transport path with a ramified structure.}
\label{figure1}
\end{figure}

In this paper, we propose an optimal resource allocation problem where a
planner chooses both an optimal allocation plan as well as an associated
optimal transport path using ramified transport technology. In both
Monge-Kantorovich and ramified transport problems, one typically assumes an
exogenous fixed distribution in both sources and targets. However, in many
resource allocation practices, the distribution in either sources or targets
is not pre-determined but rather determined endogenously. For instance, in a
\textit{production allocation} problem, suppose there are $k$ factories and $%
\ell $ households located in different places in some area. The demand for
some commodity from each household is fixed. Nevertheless, the allocation of
production among factories is not pre-determined but rather depends on the
distribution of demand among households as well as their relative locations
to factories. A planner needs to make an efficient allocation plan of
production over these $k$ factories to meet given demands from $\ell $
households. Under ramified optimal transportation, the transportation cost
of each production plan is determined by an associated optimal transport
path from factories to households. Consequently, the planner needs to find
an optimal production plan as well as an associated optimal transport path
to minimize overall cost of distributing commodity from factories to
households. Another example of similar nature exists in the following
\textit{storage arrangement} problem. Suppose there are $k$ warehouses and $%
\ell $ factories located in different places in some area. Each factory has
already produced some amount of commodity. However, the assignment of
commodity among warehouses is not pre-determined but instead relies on the
distribution of production among factories as well as relative locations
between factories and warehouses. Similarly, a planner needs to make an
efficient storage arrangement as well as an associated optimal transport
path for storing the produced commodity in the given $k$ warehouses with
minimal transportation cost.

Problem of this category is formulated as the \textit{ramified optimal
allocation problem} in Section 2. Throughout the following context, we will
focus our discussion on the scenario of the production allocation problem.
Little additional effort is needed to interpret results for other scenarios.
We start with modeling a transport path from factories to households as a
weighted directed graph, where the transportation cost on each edge of the
graph depends linearly on the length of the edge but concavely on the amount
of commodity moved on the edge. The motivation of concavity of the cost
functional on quantity comes from the observation of transport economy of
scale in group transportation. The more concave is the cost functional or
the greater is the magnitude of transport economy of scale, the more
efficient is to transport commodity in larger groups. We define the cost of
an allocation plan as the minimum transportation cost of a transport path
compatible with this plan. A planner needs to find an efficient allocation
plan such that demands from households will be met in a least cost way. In
this problem, the distribution of production over factories is not
pre-determined as in Monge-Kantorovich or ramified transport problems, but
endogenously determined by the distribution of demands from households as
well as their relative locations to factories.

We prove the existence of the ramified allocation problem in Section 3. It's
shown that due to the transport economy of scale in ramified transportation,
under any optimal allocation plan, no two factories will be connected on any
associated optimal allocation path. Consequently, any optimal allocation
path can be decomposed into a set of mutually disjoint transport paths
originating from each factory. As a result, each household will receive her
commodity from only one factory under any optimal allocation plan. It
implies that each optimal allocation plan corresponds equivalently to an
optimal assignment map from households to factories. Thus, solving the
ramified optimal allocation problem is equivalent to finding an optimal
assignment map. This optimal assignment map is shown to provide a partition
not only in households but also in the associated allocation path according
to the factories.

Because of the equivalence between the optimal allocation plan and
assignment map, we can instead focus attention on studying the properties of
optimal assignment maps in the ramified optimal allocation problem. In
Section 4, we develop a method of marginal transportation analysis to study
properties of optimal assignment maps. This method extends the standard
marginal analysis in economics into the analysis for transport paths. It
builds upon an intuitive idea that a marginal change on an optimal
allocation path should not reduce the existing minimal transportation cost.
Using this method, we develop a criterion which relates the optimal
assignment of a household with her relative location to factories and other
households, as well as her demand and productions at factories. In
particular, it is shown that each factory has a nearby region such that a
household living at this region will be assigned to the factory, where the
size of this region depends positively on the demand of the household. In
this case, the planner takes advantage of relative spatial locations between
households and factories. Also, if an optimal assignment map assigns a
household to some factory, then this household has a neighborhood area such
that any household with a smaller demand living in this area will also be
assigned to the same factory. Here, the planner utilizes the benefit in
group transportation due to transport economy of scale embedded in ramified
transportation. The role of spatial location and group transportation in
resource allocation is further studied in Section 5 by a method of
projectional analysis. We show that under an optimal assignment map, a
household will be assigned to some factory only when either she lives close
to the factory or she has some nearby neighbors assigned to the factory. In
particular, there is an \textquotedblleft autarky\textquotedblright\
situation when households and factories are located on two disjoint areas
lying distant away from each other, the demand of households will solely be
satisfied from local factories.

An important application of the properties of optimal assignment maps is
that they can shed light on the search for those maps. In Section 6, we
develop a search method utilizing these properties in a notion of state
matrix. A state matrix represents the information set of a planner during
the search process for an optimal assignment map. Any zero entry $u_{sh}$ in
the matrix reflects that the planner has excluded the possibility of
assigning household $h$ to factory $s$ under this map. When a state matrix
has exactly one non-zero entry in each column, it completely determines an
optimal assignment map by those non-zero entries. Our search method uses
properties about optimal assignment maps to update some non-zero entries
with zeros in a state matrix. This method is motivated by the observation
that via group transportation under ramified transport technology,
assignment of each household has a global effect on the allocation path as
well as the associated assignment map. Thus, the planner can deduce more
information about the optimal assignment map by exploiting the existing
information embedded in zero entries of a state matrix. Each updated state
matrix contains more zeros and thus more information than its pre-updated
counterpart. This method is useful in the search for optimal assignment maps
as each updating step increases the number of zero entries which in turn
reduces the size of the restriction set of assignment maps in a large
magnitude. In some non-trivial cases, it's shown that this method can
exactly find an optimal assignment map as desired.

\section{Ramified Optimal Allocation Problem}

In this section, we describe the setting of the optimal allocation model
with ramified optimal transportation.

\subsection{Ramified Optimal Transportation}

In a spatial economy, there are $k$ factories and $\ell $ households located
at $\mathbf{x}=\left\{ x_{1},x_{2,\cdots ,}x_{k}\right\}$ and $\left\{
y_{1},y_{2},\cdots ,y_{\ell }\right\} $ in some area $X$, where $X$ is a
compact convex subset of a Euclidean space $\mathbb{R}^{m}$. In this model
economy, there is only one commodity, and each household $j=1,\cdots ,\ell $
has a fixed demand $n_{j}>0$ for the commodity.

For analytical convenience, we first represent households and factories as
atomic Radon measures. Recall that a Radon measure $\mathbf{c}$ on $X$ is
\textit{atomic} if $\mathbf{c}$ is a finite sum of Dirac measures with
positive multiplicities, i.e.,
\begin{equation*}
\mathbf{c}=\sum\limits_{i=1}^{s}c_{i}\delta _{z_{i}}
\end{equation*}%
for some integer $s\geq 1$ and some points $z_{i}\in X$ with $c_{i}>0$ for
each $i=1,\cdots ,s$.\ The mass of $\mathbf{c}$ is denoted by
\begin{equation*}
\mathfrak{m}\left( \mathbf{c}\right) :=\sum_{i=1}^{s}c_{i}.
\end{equation*}%
We can thus represent the $\ell $ households as an atomic measure on $X$ by
\begin{equation}
\mathbf{b}=\sum\limits_{j=1}^{\ell }n_{j}\delta _{y_{j}}.  \label{households}
\end{equation}%
For each $i=1,\cdots ,k,$ denote $m_{i}$ as the units of the commodity
produced at factory $i$ located at $x_{i}$. Then, the $k$ factories can be
represented by another atomic measure on $X$ by
\begin{equation}
\mathbf{a}=\sum_{i=1}^{k}m_{i}\delta _{x_{i}}.  \label{factories}
\end{equation}%
In the study of transport problems, we usually assume $\mathfrak{m}\left(
\mathbf{a}\right) =\mathfrak{m}\left( \mathbf{b}\right) $, i.e.,
\begin{equation*}
\sum\limits_{i=1}^{k}m_{i}=\sum\limits_{j=1}^{\ell }n_{j},
\end{equation*}%
which simply means that supply equals demand in aggregate.

Next, we introduce the concept of transport path from $\mathbf{a}$ to $%
\mathbf{b}$ as in Xia \cite{xia1}.

\begin{definition}
Suppose $\mathbf{a}$\textit{\ and }$\mathbf{b}$ are two atomic measures on $%
X $ of equal mass. A \textit{transport path from }$\mathbf{a}$\textit{\ to }$%
\mathbf{b}$ is a weighted directed graph $G$ consisting of a vertex set $%
V\left( G\right) $, a directed edge set $E\left( G\right) $ and a weight
function $w:E\left( G\right) \rightarrow \left( 0,+\infty \right) $ such
that $\{x_{1},x_{2},...,x_{k}\}\cup \{y_{1},y_{2},...,y_{\ell }\}\subseteq
V(G)$ and for any vertex $v\in V(G)$, there is a balance equation
\begin{equation}
\sum_{e\in E(G),e^{-}=v}w(e)=\sum_{e\in E(G),e^{+}=v}w(e)+\left\{
\begin{array}{c}
m_{i},\text{\ if }v=x_{i}\text{\ for some }i=1,...,k \\
-n_{j},\text{\ if }v=y_{j}\text{\ for some }j=1,...,\ell \\
0,\text{\ otherwise }%
\end{array}%
\right.  \label{balance}
\end{equation}%
where each edge $e\in E\left( G\right) $ is a line segment from the starting
endpoint $e^{-}$ to the ending endpoint $e^{+}$. Denote $Path\left( \mathbf{%
a,b}\right) $ as the space of all transport paths from $\mathbf{a}$ to $%
\mathbf{b}$.
\end{definition}

Note that the balance equation (\ref{balance}) simply means the conservation
of mass at each vertex. Viewing $G$ as an one dimensional polyhedral chain,
equation (\ref{balance}) may simply be expressed as $\partial G=\mathbf{b}-%
\mathbf{a}$.

Now, we consider the transportation cost of a transport path. As observed in
both nature and efficiently designed transport networks, there exists a
\textit{transport economy of scale} underlying group transportation. For
this consideration, ramified optimal transport theory uses a cost functional
depending concavely on quantity and defines the transportation cost of a
transport path as follows.

\begin{definition}
For each transport path $G\in Path\left( \mathbf{a,b}\right) $ and any $%
\alpha \in \left[ 0,1\right] $, the $\mathbf{M}_{\alpha }$ cost of $G$ is
defined by
\begin{equation}
\mathbf{M}_{\alpha }\left( G\right) :=\sum_{e\in E\left( G\right) }\left[
w\left( e\right) \right] ^{\alpha }length\left( e\right) .  \label{M_a_cost}
\end{equation}
\end{definition}

The parameter $\alpha $ represents the magnitude of transport economy of
scale. The smaller the $\alpha ,$ the more efficient is to move commodity in
groups. Ramified optimal transport problem studies how to find a transport
path to minimize the $\mathbf{M}_{\alpha }$\textbf{\ }cost, i.e.,
\begin{equation}
\min_{G\in Path\left( \mathbf{a,b}\right) }\mathbf{M}_{\alpha }\left(
G\right) ,  \label{Ramified transport problem}
\end{equation}%
whose minimizer is called an \textit{optimal transport path }from $\mathbf{a}
$ to $\mathbf{b}$. An optimal transport path has many nice properties. For
instance, it contains no cycles by Xia \cite[Proposition 2.1]{xia1}. Thus,
without loss of generality, we assume that \textit{all transport paths
considered in the following context contain no cycles}. When $\alpha <1$, an
optimal transport path is generally of branching structure. In the scenario
with two sources and one target, a \textquotedblleft
Y-shaped\textquotedblright\ path is usually more preferable than a
\textquotedblleft V-shaped\textquotedblright\ path.

\begin{figure}[tbp]
%\begin{centering}
\includegraphics[width=5cm]{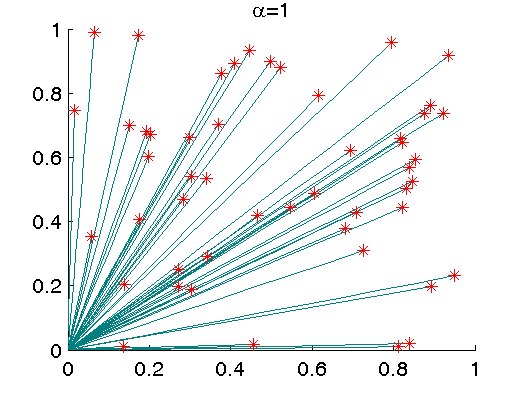}\includegraphics[width=5cm]{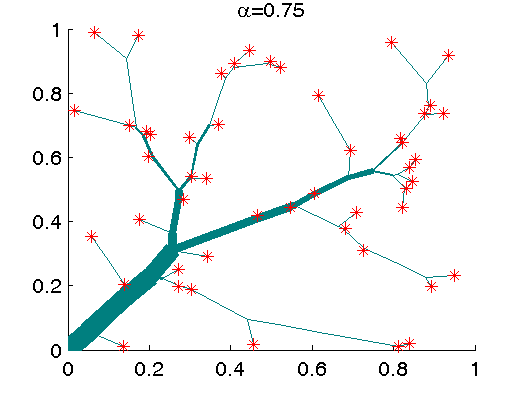}%
\includegraphics[width=5cm]{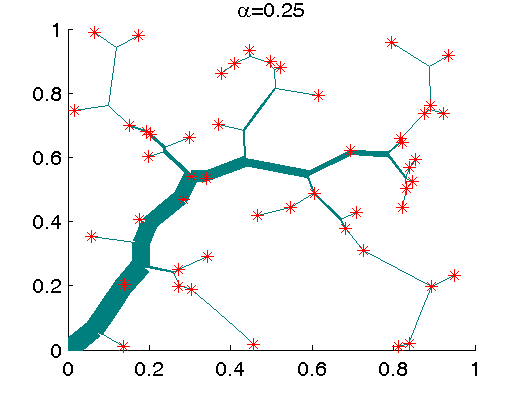} %\end{centering}
\caption{Examples of optimal transport paths.}
\label{figure2}
\end{figure}

The following example illustrates the effect of transport economy of scale
on optimal transport path in a spatial economy with one factory $\mathbf{a}%
=\delta _{O}$ located at origin and fifty households $\mathbf{b}%
=\sum_{j=1}^{50}\frac{1}{50}\delta _{y_{j}}$ of equal demand $n_{j}=\frac{1}{%
50}.$ The locations of these fifty households are randomly selected. As seen
in Figure \ref{figure2}, when $\alpha =1,$ the optimal transport path is
\textquotedblleft linear\textquotedblright\ in the sense that the factory
will ship commodity directly to each household. When $\alpha <1$, the
transport path becomes \textquotedblleft ramified\textquotedblright\ as the
planner would like commodity to be transported in groups in order to utilize
the benefit of transport economy of scale. Furthermore, by comparing for
instance the width of the transport paths for $\alpha =0.75$ and $\alpha
=0.25$, we observe that the smaller the $\alpha $, the more likely the
commodity will be transported in a large scale.

For any atomic measures $\mathbf{a}$ and $\mathbf{b}$ on $X$ of equal mass,
define the minimum transportation cost as
\begin{equation}
d_{\alpha }\left( \mathbf{a,b}\right) :=\min \left\{ \mathbf{M}_{\alpha
}\left( G\right) :G\in Path\left( \mathbf{a,b}\right) \right\} .
\label{d_alpha}
\end{equation}%
As shown in Xia \cite{xia1}, $d_{\alpha }$ is indeed a metric on the space
of atomic measures of equal mass. Also, for each $\lambda >0$, it holds that
\begin{equation*}
d_{\alpha }\left( \lambda \mathbf{a},\lambda \mathbf{b}\right) =\lambda
^{\alpha }d_{\alpha }\left( \mathbf{a,b}\right) .
\end{equation*}

Without loss of generality, we normalize both $\mathbf{a}$ and $\mathbf{b}$
to be a probability measure on $X$, i.e.,
\begin{equation}
\sum\limits_{i=1}^{k}m_{i}=\sum\limits_{j=1}^{\ell }n_{j}=1.
\label{market clearing}
\end{equation}

\subsection{Compatibility between Transport Plan and Path}

For the allocation problem, a key decision a planner needs to make is about
the transport plan from factories to households.

\begin{definition}
Suppose $\mathbf{a}$ and $\mathbf{b}$ are two atomic probability measures on
$X$ as in (\ref{factories}), (\ref{households}) and (\ref{market clearing}).
A \textit{transport plan} from $\mathbf{a}$ to $\mathbf{b}$ is an atomic
probability measure
\begin{equation}
q=\sum_{i=1}^{k}\sum_{j=1}^{\ell }q_{ij}\delta _{\left( x_{i},y_{j}\right) }
\label{transport_plan}
\end{equation}%
on the product space $X\times X$ such that for each $i$ and $j$, $q_{ij}\geq
0$,
\begin{equation}
\sum_{i=1}^{k}q_{ij}=n_{j}\text{ and }\sum_{j=1}^{\ell }q_{ij}=m_{i}.
\label{margins}
\end{equation}%
Denote $Plan\left( \mathbf{a},\mathbf{b}\right) $ as the space of all
transport plans from $\mathbf{a}$ to $\mathbf{b}$.
\end{definition}

In a transport plan $q$, the number $q_{ij}$ denotes the amount of commodity
received by household $j$ from factory $i$.

Now, as in Section 7.1 of Xia \cite{xia1}, we want to consider the
compatibility between a transport path and a transport plan. Let $G$ be a
given transport path in $Path\left( \mathbf{a,b}\right) $. Since $G$
contains no cycles, for each $x_{i}$ and $y_{j}$, there exists at most one
directed polyhedral curve $g_{ij}$ on $G$ from $x_{i}$ to $y_{j}$. In other
words, there exists a list of distinct vertices
\begin{equation}
V\left( g_{ij}\right) :=\left\{ v_{i_{1}},v_{i_{2}},\cdots ,v_{i_{h}}\right\}
\label{V_g}
\end{equation}%
in $V\left( G\right) $ with $x_{i}=v_{i_{1}}$, $y_{j}=v_{i_{h}}$, and each $%
\left[ v_{i_{t}},v_{i_{t+1}}\right] $ is a directed edge in $E\left(
G\right) $ for each $t=1,2,\cdots ,h-1$. For some pairs of $\left(
i,j\right) $, such a curve $g_{ij}$ from $x_{i}$ to $y_{j}$ may not exist,
in which case we set $g_{ij}=0$ to denote the \textit{empty} directed
polyhedral curve. By doing so, we construct a matrix
\begin{equation}
g=\left( g_{ij}\right) _{k\times \ell }  \label{g_matrix}
\end{equation}%
with each element of $g$ being a polyhedral curve. For any transport path $%
G\in Path\left( \mathbf{a,b}\right) $, such a matrix $g=\left( g_{ij}\right)
$ is uniquely determined.

\begin{definition}
Let $G\in Path\left( \mathbf{a,b}\right) $ be a transport path and $q\in
Plan\left( \mathbf{a,b}\right) $ be a transport plan. The pair $\left(
G,q\right) $ is compatible if $q_{ij}=0$ whenever $g_{ij}=0$ and
\begin{equation}
G=q\cdot g.  \label{compatible_pair}
\end{equation}
\end{definition}

Here, equation (\ref{compatible_pair}) means that as polyhedral chains,
\begin{equation*}
G=\sum_{i=1}^{k}\sum_{j=1}^{\ell }q_{ij}\cdot g_{ij},
\end{equation*}%
where the product $q_{ij}\cdot g_{ij}$ denotes that an amount $q_{ij}$ of
commodity is moved along the polyhedral curve $g_{ij}$ from factory $i$ to
household $j$.

\begin{figure}[h]
\centering
\subfloat[$G_1$]{\label{g_1}\includegraphics[width=0.35\textwidth,
height=2in]{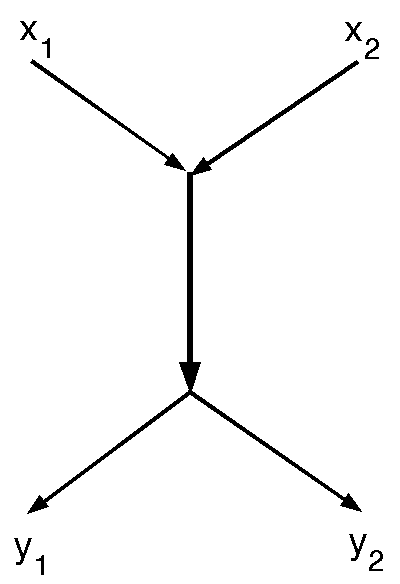}} \hspace{0.5in} \subfloat[$G_2$]{\label{g_2}%
\includegraphics[width=0.35\textwidth,height=2in]{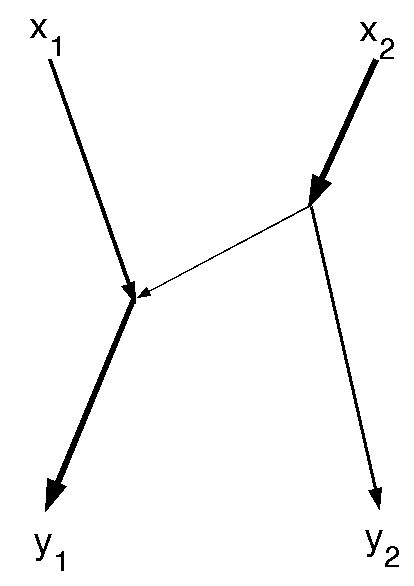}}
\caption{Compatibility between transport plan and transport path.}
\label{compatible_figure}
\end{figure}

Roughly speaking, the compatibility conditions check whether a transport
plan is realizable by a transport path. Given a transport plan, the planner
must design a transport path which can support this plan. To see the concept
more precisely, let
\begin{equation*}
\mathbf{a=}\frac{1}{4}\delta _{x_{1}}+\frac{3}{4}\delta _{x_{2}}\text{ and }%
\mathbf{b=}\frac{5}{8}\delta _{y_{1}}+\frac{3}{8}\delta _{y_{2}},
\end{equation*}%
and consider a transport plan
\begin{equation}
q=\frac{1}{8}\delta _{\left( x_{1},y_{1}\right) }+\frac{1}{8}\delta _{\left(
x_{1},y_{2}\right) }+\frac{1}{2}\delta _{\left( x_{2},y_{1}\right) }+\frac{1%
}{4}\delta _{\left( x_{2},y_{2}\right) }\in Plan\left( \mathbf{a,b}\right) .
\label{example_q}
\end{equation}%
It is straight forward to see from Figure \ref{compatible_figure} that $q$
is compatible with $G_{1}$ but not $G_{2}.$ This is because there is no
directed curve $g_{12}$ from factory $1$ to household $2$ in $G_{2}$.

\subsection{Ramified Allocation Problem}

In the standard transport problems, e.g. Monge-Kantorovich or ramified
optimal transportation, one typically assumes an exogenous fixed
distribution of production among factories. In this paper, we consider a
scenario where this distribution is not fixed but endogenously determined.
In other words, the atomic measure $\mathbf{a}$ which represents the $k$
factories can have varying production level $m_{i}$ at each factory $i$.
This consideration is motivated by our observation that in many allocation
practices as discussed in the introduction, the distribution of production
among factories is not pre-determined but rather depends on the distribution
of demand among households as well as their relative locations to factories.

\begin{definition}
Let $\mathbf{x}=\left\{ x_{1},x_{2,\cdots ,}x_{k}\right\} $ be a finite
subset of $X$, and $\mathbf{b}$ be the atomic\ probability measure
representing households defined in (\ref{households}). An allocation plan
from $\mathbf{x}$ to $\mathbf{b}$ is a probability measure
\begin{equation*}
q=\sum_{i=1}^{k}\sum_{j=1}^{\ell }q_{ij}\delta _{\left( x_{i},y_{j}\right) }
\end{equation*}%
on $X\times X$ such that $q_{ij}\geq 0$ for each $i$, $j$ and
\begin{equation*}
\sum_{i=1}^{k}q_{ij}=n_{j}\text{ for each }j=1,\cdots ,\ell .
\end{equation*}%
Denote $Plan\left[ \mathbf{x},\mathbf{b}\right] $ as the set of all
allocation plans from $\mathbf{x}$ to $\mathbf{b}$.
\end{definition}

Note that any allocation plan $q\in Plan\left[\mathbf{x},\mathbf{b}\right] $
corresponds to a transport plan $q$ from $\mathbf{a}\left( q\right) $ to $%
\mathbf{b}$, where $\mathbf{a}\left( q\right) $ is the probability measure
representing $k$ factories defined as%
\begin{equation}
\mathbf{a}\left( q\right) :=\sum_{i=1}^{k}m_{i}\left( q\right) \delta
_{x_{i}},\text{ with }m_{i}\left( q\right) =\sum_{j=1}^{\ell }q_{ij},\text{ }%
i=1,...,k.  \label{factories_vary}
\end{equation}%
In other words, $Plan\left[ \mathbf{x},\mathbf{b}\right] $ is the union of $%
Plan\left( \mathbf{a},\mathbf{b}\right) $ among all atomic probability
measures $\mathbf{a}$ supported on $\mathbf{x}$.

\begin{example}
Any function $S:\left\{ 1,\cdots ,\ell \right\} \rightarrow \left\{ 1,\cdots
,k\right\} $ determines an allocation plan in $Plan\left[\mathbf{x},\mathbf{b%
}\right] $ as
\begin{equation*}
q_{S}=\sum_{i=1}^{k}\sum_{j=1}^{\ell }q_{ij}\delta _{\left(
x_{i},y_{j}\right) }\text{ with }q_{ij}=\left\{
\begin{array}{cc}
n_{j}, & \text{if }i=S\left( j\right) \\
0, & \text{else}%
\end{array}%
\right. .
\end{equation*}%
That is,
\begin{equation}
q_{S}=\sum_{j=1}^{\ell }n_{j}\delta _{\left( x_{\left( S\left( j\right)
\right) },y_{j}\right) }.  \label{q_S}
\end{equation}
\end{example}

For a given allocation plan, we define the associated transportation cost as
follows.

\begin{definition}
For any allocation plan $q\in Plan\left[ \mathbf{x},\mathbf{b}\right] $ and $%
\alpha \in \lbrack 0,1)$, the ramified transportation cost of $q$ is
\begin{equation}
\mathbf{T}_{\alpha }\left( q\right) :=\min \left\{ \mathbf{M}_{\alpha
}\left( G\right) :G\in Path\left( \mathbf{a}\left( q\right) ,\mathbf{b}%
\right) \text{, }\left( G,q\right) \text{ compatible}\right\} ,
\label{m_alpha_q}
\end{equation}%
where $\mathbf{M}_{\alpha }\left( \cdot \right) $ is defined in (\ref%
{M_a_cost}). An \textit{allocation plan} $q^{\ast }\in Plan\left[ \mathbf{x},%
\mathbf{b}\right] $ is \textit{optimal }if
\begin{equation*}
\mathbf{T}_{\alpha }\left( q^{\ast }\right) \leq \mathbf{T}_{\alpha }\left(
q\right) \text{ for any }q\in Plan\left[ \mathbf{x},\mathbf{b}\right] .
\end{equation*}
\end{definition}

For each allocation plan $q,$ as in Xia \cite[Proposition 7.3]{xia1}, there
exists a path $G_{q}\in Path\left( \mathbf{a}\left( q\right) ,\mathbf{b}%
\right) $ such that $G_{q}$ is compatible with $q$ and
\begin{equation}
\mathbf{T}_{\alpha }\left( q\right) =\mathbf{M}_{\alpha }\left( G_{q}\right)
.  \label{planpath}
\end{equation}%
Thus, the minimum value in (\ref{m_alpha_q}) is achieved by $G_{q}$ for each
$q$. Now, we are ready to define the major problem in this paper.

\begin{problem}
\label{Problem}(Ramified Optimal Allocation Problem) Let $X$ be a compact
convex domain in $\mathbb{R}^{m}$ with the standard norm $\left\Vert \cdot
\right\Vert $. Given a finite subset $\mathbf{x}=\left\{ x_{1},x_{2,\cdots
,}x_{k}\right\} $ in $X$, an atomic probability measure $\mathbf{b}$ on $X$
defined in (\ref{households}), and a parameter $\alpha \in \lbrack 0,1)$.
Find a minimizer of $\mathbf{T}_{\alpha }\left( q\right) $ among all
allocation plans $q\in Plan\left[ \mathbf{x},\mathbf{b}\right] $, i.e.,
\begin{equation}
\min \text{ }\left\{ \mathbf{T}_{\alpha }\left( q\right) :q\in Plan\left[
\mathbf{x},\mathbf{b}\right] \right\} .  \label{Assignment
problem}
\end{equation}
\end{problem}

\section{Characterizing Optimal Allocation Plans}

In this section, we first establish the existence result of the ramified
optimal allocation problem. It's then shown that any optimal allocation plan
corresponds to an optimal assignment map from households to factories, which
provides a partition in both households and transport paths.

The following proposition proves the existence of Problem \ref{Problem}.

\begin{proposition}
\label{existence}The ramified optimal allocation problem (\ref{Assignment
problem}) has a solution. Moreover,
\begin{equation}
\min \text{ }\left\{ \mathbf{T}_{\alpha }\left( q\right) :q\in Plan\left[
\mathbf{x},\mathbf{b}\right] \right\} =\min \text{ }\left\{ d_{\alpha
}\left( \sum_{i=1}^{k}m_{i}\delta _{x_{i}}\mathbf{,b}\right) :\vec{m}\in
K\right\} ,  \label{equality}
\end{equation}%
where $K=\left\{ \vec{m}=\left( m_{1},...,m_{k}\right) \in \mathbb{R}%
_{+}^{k}:\sum_{i=1}^{k}m_{i}=1\right\} $ and $d_{\alpha }$ is defined in (%
\ref{d_alpha}).
\end{proposition}

\begin{proof}
Since $d_{\alpha }$ is a metric, it is easy to see that $d_{\alpha }\left(
\sum_{i=1}^{k}m_{i}\delta _{x_{i}}\mathbf{,b}\right) $ is a continuous
function in $\left( m_{1},\cdots ,m_{k}\right) $ on the compact subset $K$
of $%
%TCIMACRO{\U{211d} }%
%BeginExpansion
\mathbb{R}
%EndExpansion
_{+}^{k}$. Thus, the function $d_{\alpha }\left( \sum_{i=1}^{k}m_{i}\delta
_{x_{i}}\mathbf{,b}\right) $ achieves its minimum value in $K$ at some $\vec{%
m}^{\ast }=\left( m_{1}^{\ast },...,m_{k}^{\ast }\right) $. That is,
\begin{equation}
d_{\alpha }\left( \mathbf{a}^{\ast }\mathbf{,b}\right) =\min_{\vec{m}\in K}%
\text{ }d_{\alpha }\left( \sum_{i=1}^{k}m_{i}\delta _{x_{i}}\mathbf{,b}%
\right) ,  \label{d_a_min}
\end{equation}%
where $\mathbf{a}^{\ast }=\sum_{i=1}^{k}m_{i}^{\ast }\delta _{x_{i}}$. Now,
let $G^{\ast }$ be an optimal transport path in $Path\left( \mathbf{a}^{\ast
}\mathbf{,b}\right) $ with $\mathbf{M}_{\alpha }\left( G^{\ast }\right)
=d_{\alpha }\left( \mathbf{a}^{\ast }\mathbf{,b}\right) $. By Xia \cite[%
Lemma 7.1]{xia1}, $G^{\ast }$ has at least one compatible plan $q^{\ast }\in
Plan\left( \mathbf{a}^{\ast }\mathbf{,b}\right) \subseteq Plan\left[ \mathbf{%
x},\mathbf{b}\right] $ and thus
\begin{equation}
\mathbf{T}_{\alpha }\left( q^{\ast }\right) \leq \mathbf{M}_{\alpha }\left(
G^{\ast }\right) =d_{\alpha }\left( \mathbf{a}^{\ast }\mathbf{,b}\right) .
\label{inequality 2}
\end{equation}%
For any $q\in Plan\left[ \mathbf{x},\mathbf{b}\right] $, we have
\begin{eqnarray*}
\mathbf{T}_{\alpha }\left( q\right) &\geq &d_{\alpha }\left( \mathbf{a}%
\left( q\right) ,\mathbf{b}\right) \text{, by (\ref{d_alpha}) and (\ref%
{m_alpha_q})} \\
&\geq &d_{\alpha }\left( \mathbf{a}^{\ast }\mathbf{,b}\right) \text{, by (%
\ref{d_a_min})} \\
&\geq &\mathbf{T}_{\alpha }\left( q^{\ast }\right) \text{, by (\ref%
{inequality 2}).}
\end{eqnarray*}%
This shows that
\begin{equation}
\mathbf{T}_{\alpha }\left( q^{\ast }\right) =\min \text{ }\left\{ \mathbf{T}%
_{\alpha }\left( q\right) :q\in Plan\left[ \mathbf{x},\mathbf{b}\right]
\right\} \text{ and }\mathbf{T}_{\alpha }\left( q^{\ast }\right) =d_{\alpha
}\left( \mathbf{a}^{\ast }\mathbf{,b}\right) .  \label{T=d_a}
\end{equation}%
Thus, $q^{\ast }$ is a solution to the ramified optimal allocation problem (%
\ref{Assignment problem}).
\end{proof}

One implication of the above proposition is that there exists a close
relationship between optimal allocation plans and underlying transport
paths. To further characterize the properties of an optimal allocation plan,
we introduce the concept of allocation paths as follows:

\begin{definition}
An allocation path from $\mathbf{x}$ to $\mathbf{b}$ is a transport path $%
G\in Path\left( \mathbf{a,b}\right) $ for some atomic probability measure $%
\mathbf{a}$ supported on $\mathbf{x}$. Denote $Path\left[ \mathbf{x},\mathbf{%
b}\right] $ as the set of all allocation paths from $\mathbf{x}$ to $\mathbf{%
b}$. An allocation path $G^{\ast }\in Path\left[ \mathbf{x},\mathbf{b}\right]
$ is optimal if
\begin{equation*}
\mathbf{M}_{\alpha }\left( G^{\ast }\right) \leq \mathbf{M}_{\alpha }\left(
G\right) \text{ for any }G\in Path\left[ \mathbf{x},\mathbf{b}\right] .
\end{equation*}
\end{definition}

By the definition of $d_{\alpha }$, equation (\ref{equality}) can be
alternatively written as%
\begin{equation}
\min \text{ }\left\{ \mathbf{T}_{\alpha }\left( q\right) :q\in Plan\left[
\mathbf{x},\mathbf{b}\right] \right\} =\min \left\{ \mathbf{M}_{\alpha
}\left( G\right) :G\in Path\left[ \mathbf{x},\mathbf{b}\right] \right\} ,
\label{equalityTM}
\end{equation}%
which shows that the ramified optimal allocation problem corresponds to a
problem of finding an optimal allocation path. The following lemma
establishes the connection between optimal allocation plans and paths via
the notion of compatibility.

\begin{lemma}
\label{Path_Plan_Lemma}Suppose $G\in Path\left[ \mathbf{x},\mathbf{b}\right]
$, $q\in Plan\left[ \mathbf{x},\mathbf{b}\right] $ and $\left( G,q\right) $
is compatible. Then, $G$ is an optimal allocation path if and only if $q$ is
an optimal allocation plan with $\mathbf{T}_{\alpha }\left( q\right) =%
\mathbf{M}_{\alpha }\left( G\right) $.
\end{lemma}

\begin{proof}
If $G$ is an optimal allocation path, then by (\ref{equalityTM}),
\begin{equation*}
\mathbf{M}_{\alpha }\left( G\right) =\min \left\{ \mathbf{T}_{\alpha }\left(
\tilde{q}\right) :\tilde{q}\in Plan\left[ \mathbf{x},\mathbf{b}\right]
\right\} .
\end{equation*}%
Since $\left( G,q\right) $ is compatible, $\mathbf{T}_{\alpha }\left(
q\right) \leq \mathbf{M}_{\alpha }\left( G\right) $. Thus,
\begin{equation*}
\mathbf{T}_{\alpha }\left( q\right) =\min \left\{ \mathbf{T}_{\alpha }\left(
\tilde{q}\right) :\tilde{q}\in Plan\left[ \mathbf{x},\mathbf{b}\right]
\right\} =\mathbf{M}_{\alpha }\left( G\right) ,
\end{equation*}%
and $q$ is an optimal allocation plan.

Suppose $q$ is an optimal allocation plan with $\mathbf{T}_{\alpha }\left(
q\right) =\mathbf{M}_{\alpha }\left( G\right) $, then
\begin{equation*}
\mathbf{M}_{\alpha }\left( G\right) =\mathbf{T}_{\alpha }\left( q\right)
=\min \left\{ \mathbf{T}_{\alpha }\left( \tilde{q}\right) :\tilde{q}\in Plan%
\left[ \mathbf{x},\mathbf{b}\right] \right\} .
\end{equation*}%
By (\ref{equalityTM}), $G$ is an optimal allocation path.
\end{proof}

The next lemma presents a key property of an optimal allocation path.

\begin{lemma}
\label{decomposition_lemma}Let $G\in Path\left[ \mathbf{x},\mathbf{b}\right]
$ be an optimal allocation path from $\mathbf{x}$ to $\mathbf{b}$. Then, for
any $i\neq s\in \left\{ 1,\cdots ,k\right\} $, $x_{i}$ and $x_{s}$ do not
belong to the same connected component of $G$.
\end{lemma}

\begin{proof}
Assume $x_{i}$ and $x_{s}$ belong to the same connected component of an
optimal allocation path $G=\left\{ V\left( G\right) ,E\left( G\right)
,w:E\left( G\right) \rightarrow \left( 0,+\infty \right) \right\} $, then
there exists a polyhedra curve $\gamma $ supported on $G$ from $x_{i}$ to $%
x_{s}$. We may list edges of $\gamma $ as
\begin{equation*}
\left\{ \varepsilon _{1}e_{1},\cdots ,\varepsilon _{n}e_{n}\right\} \text{
with }\varepsilon _{i}=\pm 1\text{ and }e_{i}\in E\left( G\right) \text{.}
\end{equation*}%
Here, $\varepsilon _{i}=1$ (or $-1$) if $e_{i}$ has the same (or opposite)
direction as $\gamma $. Let
\begin{equation*}
\lambda =\min_{1\leq i\leq n}w\left( e_{i}\right) >0,
\end{equation*}%
and consider $G_{t}:=G+t\gamma $ for $t=\pm \lambda $. Note that $G_{t}$ is
still in $Path\left[ \mathbf{x},\mathbf{b}\right] $, and%
\begin{equation*}
\mathbf{M}_{\alpha }\left( G_{t}\right) -\mathbf{M}_{\alpha }\left( G\right)
=\sum_{i=1}^{n}\left[ \left( w\left( e_{i}\right) +t\varepsilon _{i}\right)
^{\alpha }-\left( w\left( e_{i}\right) \right) ^{\alpha }\right]
length\left( e_{i}\right) .
\end{equation*}%
Thus,
\begin{eqnarray*}
&&\mathbf{M}_{\alpha }\left( G_{\lambda }\right) +\mathbf{M}_{\alpha }\left(
G_{-\lambda }\right) -2\mathbf{M}_{\alpha }\left( G\right) \\
&=&\sum_{i=1}^{n}\left[ \left( w\left( e_{i}\right) +\lambda \right)
^{\alpha }+\left( w\left( e_{i}\right) -\lambda \right) ^{\alpha }-2\left(
w\left( e_{i}\right) \right) ^{\alpha }\right] length\left( e_{i}\right) .
\end{eqnarray*}%
When $\alpha \in \left( 0,1\right) $, by the strict concavity of $x^{\alpha
} $, we have
\begin{equation*}
\left( w\left( e_{i}\right) +\lambda \right) ^{\alpha }+\left( w\left(
e_{i}\right) -\lambda \right) ^{\alpha }-2\left( w\left( e_{i}\right)
\right) ^{\alpha }<0.
\end{equation*}%
When $\alpha =0$,
\begin{eqnarray*}
&&\sum_{i=1}^{n}\left[ \left( w\left( e_{i}\right) +\lambda \right) ^{\alpha
}+\left( w\left( e_{i}\right) -\lambda \right) ^{\alpha }-2\left( w\left(
e_{i}\right) \right) ^{\alpha }\right] length\left( e_{i}\right) \\
&=&-\sum \left\{ length\left( e_{i}\right) :w\left( e_{i}\right) =\lambda
\right\} <0.
\end{eqnarray*}%
Thus, when $\alpha \in \lbrack 0,1)$, we have
\begin{equation*}
\mathbf{M}_{\alpha }\left( G_{\lambda }\right) +\mathbf{M}_{\alpha }\left(
G_{-\lambda }\right) -2\mathbf{M}_{\alpha }\left( G\right) <0.
\end{equation*}%
i.e.
\begin{equation*}
\min \left\{ \mathbf{M}_{\alpha }\left( G_{\lambda }\right) ,\mathbf{M}%
_{\alpha }\left( G_{-\lambda }\right) \right\} <\mathbf{M}_{\alpha }\left(
G\right) ,
\end{equation*}%
a contradiction with the optimality of $G$ in $Path\left[ \mathbf{x},\mathbf{%
b}\right] $.
\end{proof}

The above lemma says that no two factories will be connected on any optimal
allocation path. Alternatively speaking, on an optimal allocation path, each
single household will receive her commodity from only one factory, i.e.,
each household is assigned to one factory. This result is attributed to the
transport economy of scale underlying ramified transportation technology. As
seen in Section 2, an $\alpha \in \lbrack 0,1)$ implies the existence of
transport economy of scale with transporting in groups being more cost
efficient than transporting separately. Any allocation path on which some
single household receives commodity from two factories can not be optimal
because the planner would be able to reduce transportation cost by
transferring production of one factory to the other. This transfer makes the
benefit of transport economy of scale more likely to be realized as
commodity for this household is transported in a larger scale on the path.

The result that each household is assigned to one factory on an optimal
allocation path motivates the following notion of assignment map.

\begin{definition}
An \textit{assignment map is a function }$S:\left\{ 1,\cdots ,\ell \right\}
\rightarrow \left\{ 1,\cdots ,k\right\} $. Let $Map\left[ \ell ,k\right] $
be the set of all \textit{assignment maps}. For any assignment map $S\in Map%
\left[ \ell ,k\right] $ and $\alpha \in \lbrack 0,1)$, define
\begin{equation*}
\mathbf{E}_{\alpha }\left( S;\mathbf{x},\mathbf{b}\right)
:=\sum_{i=1}^{k}d_{\alpha }\left( \mathbf{a}_{i}\mathbf{,b}_{i}\right) ,
\end{equation*}%
where $d_{\alpha }$ is the metric defined in (\ref{d_alpha}),%
\begin{equation}
\mathbf{a}_{i}=\left( \sum_{j\in S^{-1}\left( i\right) }n_{j}\right) \delta
_{x_{i}}\text{ and }\mathbf{b}_{i}=\sum_{j\in S^{-1}\left( i\right)
}n_{j}\delta _{y_{i}}.  \label{a_b_i}
\end{equation}%
An assignment map $S^{\ast }\in Map\left[ \ell ,k\right] $ is optimal if
\begin{equation*}
\mathbf{E}_{\alpha }\left( S^{\ast };\mathbf{x},\mathbf{b}\right) \leq
\mathbf{E}_{\alpha }\left( S;\mathbf{x},\mathbf{b}\right) \text{ for any }%
S\in Map\left[ \ell ,k\right] .
\end{equation*}
\end{definition}

Using the concept of assignment maps, Lemma \ref{decomposition_lemma}
provides a partition result for an optimal allocation path.

\begin{proposition}
\label{Path_map}Let $G$ be an optimal allocation path from $\mathbf{x}$ to $%
\mathbf{b}$. Then, there exists an assignment map $S\in Map\left[ \ell ,k%
\right] $ such that
\begin{equation*}
\mathbf{E}_{\alpha }\left( S;\mathbf{x},\mathbf{b}\right) =\mathbf{M}%
_{\alpha }\left( G\right) .
\end{equation*}%
Moreover, for $\mathbf{a}_{i}$ and $\mathbf{b}_{i}$ given in (\ref{a_b_i}),
the path $G$ can be decomposed into the sum of $k$ pairwise disjoint
transport paths
\begin{equation}
G=\sum_{i=1}^{k}G_{i}\text{ with }\mathbf{M}_{\alpha }\left( G\right)
=\sum_{i=1}^{k}\mathbf{M}_{\alpha }\left( G_{i}\right) ,  \label{equation_G}
\end{equation}%
where each $G_{i}\in Path\left( \mathbf{a}_{i},\mathbf{b}_{i}\right) $ is an
optimal transport path with $\mathbf{M}_{\alpha }\left( G_{i}\right)
=d_{\alpha }\left( \mathbf{a}_{i}\mathbf{,b}_{i}\right) .$
\end{proposition}

\begin{proof}
Let $G$ be an optimal allocation path from $\mathbf{x}$ to $\mathbf{b}$. By
Lemma \ref{decomposition_lemma}, each $x_{i}$ determines a connected
component $G_{i}$ of $G$. Thus,%
\begin{equation}
G=\sum_{i=1}^{k}G_{i}\text{ with }\mathbf{M}_{\alpha }\left( G\right)
=\sum_{i=1}^{k}\mathbf{M}_{\alpha }\left( G_{i}\right) .
\label{equation(M_G)}
\end{equation}

For each $j\in \left\{ 1,\cdots ,\ell \right\} $, $y_{j}$ is clearly
connected to some $x_{i}$ on $G$. By Lemma \ref{decomposition_lemma}, such
an $x_{i}$ must be unique for each $j$. Thus, we may define
\begin{equation*}
S\left( j\right) =i,\text{ if }y_{j}\text{ is connected to }x_{i}\text{ on }%
G.
\end{equation*}%
This defines an assignment map $S\in Map\left[ \ell ,k\right] $. Also note
that $G_{i}$ is a transport path from $\mathbf{a}_{i}$ to $\mathbf{b}_{i}$,
where $\mathbf{a}_{i},\mathbf{b}_{i}$ are given in (\ref{a_b_i}). Since $G$
is optimal as in (\ref{equalityTM}), each $G_{i}\in Path\left( \mathbf{a}%
_{i},\mathbf{b}_{i}\right) $ must also be optimal, and hence
\begin{equation*}
\mathbf{M}_{\alpha }\left( G_{i}\right) =d_{\alpha }\left( \mathbf{a}_{i},%
\mathbf{b}_{i}\right) .
\end{equation*}%
Therefore,
\begin{equation*}
\mathbf{M}_{\alpha }\left( G\right) =\sum_{i=1}^{k}\mathbf{M}_{\alpha
}\left( G_{i}\right) =\sum_{i=1}^{k}d_{\alpha }\left( \mathbf{a}_{i},\mathbf{%
b}_{i}\right) =\mathbf{E}_{\alpha }\left( S;\mathbf{x},\mathbf{b}\right) .
\end{equation*}
\end{proof}

\begin{figure}[h]
\label{G_1} \centering
\includegraphics[width=0.6\textwidth]{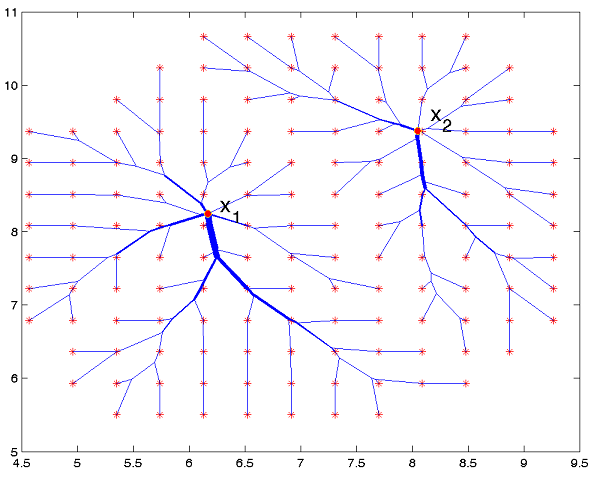}
\caption{A partition in households $\mathbf{b}$ and allocation path $G.$}
\label{partition}
\end{figure}

Figure \ref{partition} illustrates a partition result under an optimal
assignment map. This spatial economy consists of factories $\left\{
x_{1},x_{2}\right\} $ and households $\left\{ y_{1},\cdots ,y_{145}\right\} $
with equal demand $n_{j}=\frac{1}{145}$. The figure shows a clear partition
in both the allocation path and households: First, the allocation path is
decomposed into two disjoint sub-transport paths originating from factory $%
x_{1}$ and $x_{2}$ respectively. Second, the $145$ households are divided
into two unconnected populations centered around the factory from which they
receive the commodity.

Note that any assignment map $S\in Map\left[ \ell ,k\right] $ provides a
partition among households
\begin{equation*}
\mathbf{b=}\sum_{i=1}^{k}\mathbf{b}_{i},
\end{equation*}%
where all households in $\mathbf{b}_{i}$ given in (\ref{a_b_i}) are assigned
to a single factory located at $x_{i}$. For each $i$, let $G_{i}\in
Path\left( \mathbf{a}_{i}\mathbf{,b}_{i}\right) $ be an optimal transport
path which transports the commodity produced at factory $i$ to households in
$\mathbf{b}_{i}.$ This yields an allocation path from $\mathbf{x}$ to $%
\mathbf{b}$
\begin{equation}
G_{S}=\sum_{i=1}^{k}G_{i}\in Path\left[ \mathbf{x,b}\right] ,  \label{G_S}
\end{equation}%
which is compatible with the allocation plan $q_{S}$ given in (\ref{q_S}).
Thus,
\begin{equation}
\mathbf{E}_{\alpha }\left( S;\mathbf{x},\mathbf{b}\right)
=\sum_{i=1}^{k}d_{\alpha }\left( \mathbf{a}_{i}\mathbf{,b}_{i}\right)
=\sum_{i=1}^{k}\mathbf{M}_{\alpha }\left( G_{i}\right) \geq \mathbf{M}%
_{\alpha }\left( G_{S}\right) \geq \mathbf{T}_{\alpha }\left( q_{S}\right) .
\label{E_energy_greater_than_T_energy}
\end{equation}%
By Proposition \ref{Path_map}, any optimal allocation path is in the form of
(\ref{G_S}) with respect to its associated assignment map.

Now, we state the main results of this section as follows:

\begin{theorem}
\label{Theorem 1}Given a subset $\mathbf{x}=\left\{ x_{1},x_{2,\cdots
,}x_{k}\right\} $ in $X$, an atomic probability measure $\mathbf{b}$ as in (%
\ref{households}), and a parameter $\alpha \in \lbrack 0,1)$.

\begin{enumerate}
\item \label{map_plan}An allocation plan $q\in Plan\left[ \mathbf{x},\mathbf{%
b}\right] $ is optimal if and only if there exists an optimal assignment map
$S\in Map\left[ \ell ,k\right] $ such that $q=q_{S}$.

\item \label{map_path}An allocation path $G\in Path\left[ \mathbf{x},\mathbf{%
b}\right] $ is optimal if and only if there exists an optimal assignment map
$S\in Map\left[ \ell ,k\right] $ such that $G=G_{S}$ for some $G_{S}$ of $S$
defined in (\ref{G_S}).

\item Moreover,
\begin{equation}
\min_{q\in Plan\left[ \mathbf{x},\mathbf{b}\right] }\text{ }\mathbf{T}%
_{\alpha }\left( q\right) =\min_{S\in Map\left[ \ell ,k\right] }\mathbf{E}%
_{\alpha }\left( S;\mathbf{x},\mathbf{b}\right) =\min_{G\in Path\left[
\mathbf{x},\mathbf{b}\right] }\mathbf{M}_{\alpha }\left( G\right) .
\label{two_min_equal}
\end{equation}
\end{enumerate}
\end{theorem}

\begin{proof}
We first show the equivalence in (\ref{map_plan}):

\textquotedblleft $\Longrightarrow $\textquotedblright . Suppose $q\in Plan%
\left[ \mathbf{x},\mathbf{b}\right] $ is an optimal allocation plan. Let $%
G\in Path\left[ \mathbf{x},\mathbf{b}\right] $ be an allocation path that is
compatible with $q$ and $\mathbf{M}_{\alpha }\left( G\right) =\mathbf{T}%
_{\alpha }\left( q\right) .$ Since $q$ is an optimal allocation plan, by
Lemma \ref{Path_Plan_Lemma}, $G$ is an optimal allocation path. By
Proposition \ref{Path_map}, there exists an assignment map $S$ such that
\begin{equation}
\mathbf{E}_{\alpha }\left( S;\mathbf{x},\mathbf{b}\right) =\mathbf{M}%
_{\alpha }\left( G\right) =\mathbf{T}_{\alpha }\left( q\right) .
\label{E_M_T}
\end{equation}%
Since $\left( G,q\right) $ is compatible, by (\ref{compatible_pair}), we may
express
\begin{equation*}
G=\sum_{i=1}^{k}\sum_{j=1}^{\ell }q_{ij}\cdot g_{ij}.
\end{equation*}%
For each $j$, by Lemma \ref{decomposition_lemma}, we have $g_{sj}=0$
whenever $s\neq S\left( j\right) $. By the compatibility of $\left(
G,q\right) $, we have $q_{sj}=0$ whenever $s\neq S\left( j\right) $. Thus,
for $i=S\left( j\right) \,$,

\begin{equation*}
q_{ij}=\sum_{s=1}^{k}q_{sj}=n_{j}\text{, for each }j\text{.}
\end{equation*}%
This shows that $q=q_{S}$.

For any $\tilde{S}\in Map\left[ \ell ,k\right] $, by (\ref%
{E_energy_greater_than_T_energy}) and (\ref{E_M_T}),
\begin{equation*}
\mathbf{E}_{\alpha }\left( \tilde{S};\mathbf{x},\mathbf{b}\right) \geq
\mathbf{T}_{\alpha }\left( q_{\tilde{S}}\right) \geq \mathbf{T}_{\alpha
}\left( q\right) =\mathbf{E}_{\alpha }\left( S;\mathbf{x},\mathbf{b}\right) .
\end{equation*}%
Thus, $S$ is an optimal assignment map. Since $\mathbf{E}_{\alpha }\left( S;%
\mathbf{x},\mathbf{b}\right) =\mathbf{T}_{\alpha }\left( q\right) $, we have
the first equality in (\ref{two_min_equal}).

\textquotedblleft $\Longleftarrow $\textquotedblright . Suppose $S\in Map%
\left[ \ell ,k\right] $ is an optimal assignment map. Then, by (\ref%
{E_energy_greater_than_T_energy}) and the first equality in (\ref%
{two_min_equal}),
\begin{eqnarray*}
\mathbf{T}_{\alpha }\left( q_{S}\right) &\leq &\mathbf{E}_{\alpha }\left( S;%
\mathbf{x},\mathbf{b}\right) =\min \left\{ \mathbf{E}_{\alpha }\left( \tilde{%
S};\mathbf{x},\mathbf{b}\right) :\tilde{S}\in Map\left[ \ell ,k\right]
\right\} \\
&=&\min \text{ }\left\{ \mathbf{T}_{\alpha }\left( \tilde{q}\right) :\tilde{q%
}\in Plan\left[ \mathbf{x},\mathbf{b}\right] \right\} .
\end{eqnarray*}%
Therefore, $q_{S}\in Plan\left[ \mathbf{x},\mathbf{b}\right] $ is an optimal
allocation plan. This completes the proof of (\ref{map_plan}).

The second equality of (\ref{two_min_equal}) follows from the first equality
of (\ref{two_min_equal}) and (\ref{equalityTM}).

We now show the equivalence in (\ref{map_path}):

\textquotedblleft $\Longleftarrow $\textquotedblright\ Suppose $G=G_{S}$ for
some optimal assignment map $S\in Map\left[ \ell ,k\right] $, then
\begin{eqnarray*}
\mathbf{M}_{\alpha }\left( G_{S}\right) &\leq &\sum_{i=1}^{k}\mathbf{M}%
_{\alpha }\left( G_{i}\right) =\sum_{i=1}^{k}d_{\alpha }\left( \mathbf{a}_{i}%
\mathbf{,b}_{i}\right) =\mathbf{E}_{\alpha }\left( S;\mathbf{x},\mathbf{b}%
\right) \\
&=&\min_{\tilde{S}\in Map\left[ \ell ,k\right] }\mathbf{E}_{\alpha }\left(
\tilde{S};\mathbf{x},\mathbf{b}\right) =\min_{G\in Path\left[ \mathbf{x},%
\mathbf{b}\right] }\mathbf{M}_{\alpha }\left( G\right) \text{,}
\end{eqnarray*}%
by the optimality of $S$ and (\ref{two_min_equal}). Thus, $G=G_{S}$ is an
optimal allocation path.

\textquotedblleft $\Longrightarrow $\textquotedblright\ Suppose $G\in Path%
\left[ \mathbf{x},\mathbf{b}\right] $ is an optimal allocation path. By
Proposition \ref{Path_map}, $G=G_{S}$ for some $S\in Map\left[ \ell ,k\right]
$ with $\mathbf{E}_{\alpha }\left( S;\mathbf{x},\mathbf{b}\right) =\mathbf{M}%
_{\alpha }\left( G\right) $. Then, the optimality of $S$ follows from the
optimality of $G$ and (\ref{two_min_equal}).
\end{proof}

Theorem \ref{Theorem 1} shows that in the ramified optimal allocation
problem, there exists an equivalence between optimal allocation plan and
optimal assignment map. This result has an analogous counterpart in
Monge-Kantorovich problems, but has not been observed in current literature
on ramified transport problems. An implication of this theorem is that one
can instead search for an optimal assignment map in order to find an optimal
allocation plan. Each optimal assignment map $S\in Map\left[ \ell ,k\right] $
would give an optimal allocation plan $q_{S}\in Plan\left[ \mathbf{x},%
\mathbf{b}\right] $ as in (\ref{q_S}). For this consideration, the rest of
this paper will focus attention on characterizing the various properties of
optimal assignment maps.

\section{Properties of Optimal Assignment Maps via Marginal Analysis}

In this section, we develop a method of marginal transportation analysis and
use it to study the properties of optimal assignment maps.

We first formalize a concept of marginal transportation cost for a
single-source transport system. Let $G\in Path\left( \mathfrak{m}\left(
\mathbf{c}\right) \delta _{O}\mathbf{,c}\right) $ be a transport path from a
\textit{single source} $O$ to an atomic measure $\mathbf{c}$ of mass $%
\mathfrak{m}\left( \mathbf{c}\right) $. For any point $p$ on the support of $%
G$, we set
\begin{equation}
\theta \left( p\right) :=%
\begin{cases}
w(e), & \text{if }p\text{ is in the interior of some edge }e\in E\left(
G\right) \\
\mathfrak{m}\left( \mathbf{c}\right) , & \text{if }p=O \\
\sum_{e\in E(G),e^{+}=p}w(e), & \text{if }p\in V\left( G\right) \setminus
\left\{ O\right\}%
\end{cases}%
,  \label{density}
\end{equation}%
which represents the mass flowing through $p$, where $e^{+}$ denotes the
ending endpoint of edge $e\in E\left( G\right) $. Since $G$ has a single
source and contains no cycles, for any point $p$ on $G$, there exists a
unique polyhedral curve $\gamma _{p}$ on $G$ from $O$ to $p$. Moreover, for
any point $s\in \gamma _{p}$, it holds that
\begin{equation}
\mathfrak{m}\left( \mathbf{c}\right) \geq \theta \left( s\right) \geq \theta
\left( p\right) \text{.}  \label{density_comparison}
\end{equation}%
When the mass at $p$ changes by an amount $\Delta m$ with $\Delta m\geq -$ $%
\mathbf{\theta }\left( p\right) $, then the mass flowing through each point
of $\gamma _{p}$ also changes by $\Delta m$. As a result, the corresponding
increment of the transportation cost is
\begin{eqnarray}
\Delta C_{G}\left( p,\Delta m\right):&=&\mathbf{M}_{\alpha }\left( G+\left(
\Delta m\right) \gamma _{p}\right) -\mathbf{M}_{\alpha }\left( G\right)
\label{marginal_increment} \\
&=&\int_{\gamma _{p}}\left( \theta \left( s\right) +\Delta m\right) ^{\alpha
}-\left( \theta \left( s\right) \right) ^{\alpha }ds.  \notag
\end{eqnarray}%
The\textit{\ marginal transportation cost} at $p$ via $G$ is defined by
\begin{equation*}
MC_{G}\left( p\right) :=\lim_{\Delta m\rightarrow 0}\frac{\Delta C_{G}\left(
p,\Delta m\right) }{\Delta m}=\alpha \int_{\gamma _{p}}\left( \theta \left(
s\right) \right) ^{\alpha -1}ds.
\end{equation*}%
The following proposition establishes some properties of the function $%
\Delta C_{G}\left( p,\Delta m\right) $. Those properties are the key
elements of marginal transportation analysis used to study optimal
assignment maps later.

\begin{proposition}
\label{prop increment cost}For any $G\in Path\left( \mathfrak{m}\left(
\mathbf{c}\right) \delta _{O}\mathbf{,c}\right) $ and $p$ on $G$, we have
\begin{equation}
\Delta C_{G}\left( p,-\Delta m\right) =-\Delta C_{\tilde{G}}\left( p,\Delta
m\right) ,\text{ for }\Delta m\in \left[ -\theta \left( p\right) ,\theta
\left( p\right) \right] ,  \label{alternative}
\end{equation}%
where $\tilde{G}=G-\left( \Delta m\right) \gamma _{p}$. Moreover, for any $%
\Delta m\geq 0$, we have
\begin{equation}
\left[ \left( \mathfrak{m}\left( \mathbf{c}\right) +\Delta m\right) ^{\alpha
}-\mathfrak{m}\left( \mathbf{c}\right) ^{\alpha }\right] \int_{\gamma
_{p}}ds\leq \Delta C_{G}\left( p,\Delta m\right) \leq \left[ \left( \theta
\left( p\right) +\Delta m\right) ^{\alpha }-\theta \left( p\right) ^{\alpha }%
\right] \int_{\gamma _{p}}ds.  \label{incrementcost_comparison}
\end{equation}%
If, in addition, $G\in Path\left( \mathfrak{m}\left( \mathbf{c}\right)
\delta _{O}\mathbf{,c}\right) $ is optimal in (\ref{Ramified transport
problem}), then
\begin{equation}
d_{\alpha }\left( \left( \mathfrak{m}\left( \mathbf{c}\right) +\Delta
m\right) \delta _{O},\mathbf{c+}\left( \Delta m\right) \delta _{p}\right)
-d_{\alpha }\left( \mathfrak{m}\left( \mathbf{c}\right) \delta _{O},\mathbf{c%
}\right) \leq \Delta C_{G}\left( p,\Delta m\right)  \label{d_a_comparision}
\end{equation}%
for any $\Delta m\geq -$ $\mathbf{\theta }\left( p\right) $ with $p$ on $G$.
\end{proposition}

\begin{proof}
Let $\tilde{G}=G-\left( \Delta m\right) \gamma _{p}$, then when $\Delta m\in %
\left[ -\theta \left( p\right) ,\theta \left( p\right) \right] $,
\begin{eqnarray*}
\Delta C_{G}\left( p,-\Delta m\right) &=&\int_{\gamma _{p}}\left( \theta
\left( s\right) -\Delta m\right) ^{\alpha }-\left( \theta \left( s\right)
\right) ^{\alpha }ds \\
&=&-\int_{\gamma _{p}}\left[ \left( \left( \theta \left( s\right) -\Delta
m\right) +\Delta m\right) ^{\alpha }-\left( \theta \left( s\right) -\Delta
m\right) ^{\alpha }\right] ds \\
&=&-\Delta C_{\tilde{G}}\left( p,\Delta m\right) .
\end{eqnarray*}%
For any $\Delta m\geq 0$, note that the function $f\left( t\right) :=\left(
t+\Delta m\right) ^{\alpha }-t^{\alpha }$ is monotonic non-increasing on $%
t>0 $ when $0\leq \alpha <1$. By (\ref{density_comparison}), we have
\begin{equation*}
\left( \mathfrak{m}\left( \mathbf{c}\right) +\Delta m\right) ^{\alpha
}-\left( \mathfrak{m}\left( \mathbf{c}\right) \right) ^{\alpha }\leq \left(
\theta \left( s\right) +\Delta m\right) ^{\alpha }-\left( \theta \left(
s\right) \right) ^{\alpha }\leq \left( \theta \left( p\right) +\Delta
m\right) ^{\alpha }-\left( \theta \left( p\right) \right) ^{\alpha }.
\end{equation*}%
Thus, by (\ref{marginal_increment}), inequalities (\ref%
{incrementcost_comparison}) hold.

When $G$ is also optimal, we have $\mathbf{M}_{\alpha }\left( G\right)
=d_{\alpha }\left( \mathfrak{m}\left( \mathbf{c}\right) \delta _{O},\mathbf{c%
}\right) $. For any $\Delta m\geq -$ $\mathbf{\theta }\left( p\right) $,
since $G+\left( \Delta m\right) \gamma _{p}\in Path\left( \left( \mathfrak{m}%
\left( \mathbf{c}\right) +\Delta m\right) \delta _{O},\mathbf{c+}\left(
\Delta m\right) \delta _{p}\right) $ , we have
\begin{eqnarray*}
\Delta C_{G}\left( p,\Delta m\right) &=&\mathbf{M}_{\alpha }\left( G+\left(
\Delta m\right) \gamma _{p}\right) -\mathbf{M}_{\alpha }\left( G\right) \\
&\geq &d_{\alpha }\left( \left( \mathfrak{m}\left( \mathbf{c}\right) +\Delta
m\right) \delta _{O},\mathbf{c+}\left( \Delta m\right) \delta _{p}\right)
-d_{\alpha }\left( \mathfrak{m}\left( \mathbf{c}\right) \delta _{O},\mathbf{c%
}\right) .
\end{eqnarray*}
\end{proof}

Now, we apply marginal transportation analysis developed above to study
properties of an optimal assignment map. Let $G\in Path\left[ \mathbf{x},%
\mathbf{b}\right] $ be any optimal allocation path. By Theorem \ref{Theorem
1}, $G$ must be in the form of (\ref{G_S}), which is simply a disjoint union
of single-source paths $G_{i}$'s. For any $p$ on $G$, there exists a unique $%
i$ such that $p$ is on $G_{i}$. Thus, we can define the corresponding $%
\theta _{i}\left( p\right) $ and $\Delta C_{G_{i}}\left( p,\Delta m\right) $
as in (\ref{density}) and (\ref{marginal_increment}). Then, we set
\begin{equation*}
\theta \left( p\right) :=\theta _{i}\left( p\right) \text{ and }\Delta
C_{G}\left( p,\Delta m\right) :=\Delta C_{G_{i}}\left( p,\Delta m\right)
\end{equation*}%
for any $p$ on the support of $G$.

For any $\alpha \in (0,1)$ and $\sigma \geq \epsilon >0$, define%
\begin{equation}
\rho _{\alpha }\left( \sigma ,\epsilon \right) :=\left( \frac{\sigma }{%
\epsilon }\right) ^{\alpha }-\left( \frac{\sigma }{\epsilon }-1\right)
^{\alpha }\text{.}  \label{W-function}
\end{equation}%
The function $\rho _{\alpha }\left( \sigma ,\varepsilon \right) $ is
decreasing in $\sigma $ and increasing in $\epsilon $; also $0<\rho _{\alpha
}\left( \sigma ,\epsilon \right) \leq 1$ for any $\sigma \geq \epsilon >0$.
For $\alpha =0$, set $\rho _{0}\left( \sigma ,\epsilon \right) =0$ when $%
\sigma >\epsilon >0$ and $\rho _{0}\left( \sigma ,\epsilon \right) =1$ when $%
\sigma =\epsilon >0$.

\begin{proposition}
\label{prop key inequality}Suppose $G\in Path\left[ \mathbf{x},\mathbf{b}%
\right] $ is an optimal allocation path as given in (\ref{equation_G}). Let $%
p$ be a point on $G_{s}$ for some $s\in \left\{ 1,\cdots ,k\right\} $ and $%
\Delta m\in (0,\theta \left( p\right) ]$. For any $p^{\ast }$ on $G$ with $%
\gamma _{p}\cap \gamma _{p^{\ast }}$ having zero length, we have%
\begin{equation}
\Delta C_{G}\left( p,-\Delta m\right) +\left( \Delta m\right) ^{\alpha
}\left\Vert p-p^{\ast }\right\Vert +\Delta C_{G}\left( p^{\ast },\Delta
m\right) \geq 0  \label{increment_cost_inequality}
\end{equation}%
and
\begin{equation}
\left\Vert p-p^{\ast }\right\Vert +\rho _{\alpha }\left( \theta \left(
p^{\ast }\right) +\Delta m,\Delta m\right) \int_{\gamma _{p^{\ast }}}ds\geq
\rho _{\alpha }\left( \mathfrak{m}\left( \mathbf{b}_{s}\right) ,\Delta
m\right) \int_{\gamma _{p}}ds,  \label{p_p*}
\end{equation}%
where $\left\Vert \cdot \right\Vert $ stands for the standard norm on $%
\mathbb{R}^{m}$. In particular, for any $i\in \left\{ 1,\cdots ,k\right\} $,
\begin{equation}
\left\Vert p-x_{i}\right\Vert \geq \rho _{\alpha }\left( \mathfrak{m}\left(
\mathbf{b}_{s}\right) ,\Delta m\right) \int_{\gamma _{p}}ds\geq \rho
_{\alpha }\left( \mathfrak{m}\left( \mathbf{b}_{s}\right) ,\Delta m\right)
\left\Vert p-x_{s}\right\Vert .  \label{p_x_i}
\end{equation}%
Moreover, suppose $p^{\ast }$ is on $G_{i}$ with $i\neq s$ and $\Delta m\leq
\theta \left( p^{\ast }\right) $, then
\begin{equation}
\left\Vert p-p^{\ast }\right\Vert +\frac{\rho _{\alpha }\left( \theta \left(
p^{\ast }\right) +\Delta m,\Delta m\right) }{\rho _{\alpha }\left( \mathfrak{%
m}\left( \mathbf{b}_{i}\right) ,\Delta m\right) }\left\Vert p^{\ast
}-x_{i}\right\Vert \geq \rho _{\alpha }\left( \mathfrak{m}\left( \mathbf{b}%
_{s}\right) ,\Delta m\right) \left\Vert p-x_{s}\right\Vert .
\label{p_neighborhood}
\end{equation}
\end{proposition}

\begin{proof}
Let $\hat{G}=G-\left( \Delta m\right) \gamma _{p}+\left( \Delta m\right)
[p,p^{\ast }]+\left( \Delta m\right) \gamma _{p^{\ast }}$, where $[p,p^{\ast
}]$ denotes the line segment from $p$ to $p^{\ast }$. Then, when the
intersection of polyhedral curves $\gamma _{p}\cap \gamma _{p^{\ast }}$ has
length zero, we have
\begin{eqnarray*}
&&\Delta C_{G}\left( p,-\Delta m\right) +\left( \Delta m\right) ^{\alpha
}\left\Vert p-p^{\ast }\right\Vert +\Delta C_{G}\left( p^{\ast },\Delta
m\right) \\
&=&\int_{\gamma _{p}}\left[ \left( \theta \left( s\right) -\Delta m\right)
^{\alpha }-\theta \left( s\right) ^{\alpha }\right] ds+\left( \Delta
m\right) ^{\alpha }\left\Vert p-p^{\ast }\right\Vert \\
&&+\int_{\gamma _{p^{\ast }}}\left[ \left( \theta \left( s\right) +\Delta
m\right) ^{\alpha }-\theta \left( s\right) ^{\alpha }\right] ds \\
&\geq &\mathbf{M}_{\alpha }\left( \hat{G}\right) -\mathbf{M}_{\alpha }\left(
G\right) \geq 0\text{, by the optimality of }G.
\end{eqnarray*}%
To prove (\ref{p_p*}), we observe that
\begin{eqnarray*}
&&\left( \Delta m\right) ^{\alpha }\left[ \left\Vert p-p^{\ast }\right\Vert
+\rho _{\alpha }\left( \theta \left( p^{\ast }\right) +\Delta m,\Delta
m\right) \int_{\gamma _{p^{\ast }}}ds\right] \\
&=&\left( \Delta m\right) ^{\alpha }\left\Vert p-p^{\ast }\right\Vert +\left[
\left( \theta \left( p^{\ast }\right) +\Delta m\right) ^{\alpha }-\left(
\theta \left( p^{\ast }\right) \right) ^{\alpha }\right] \int_{\gamma
_{p^{\ast }}}ds \\
&\geq &\left( \Delta m\right) ^{\alpha }\left\Vert p-p^{\ast }\right\Vert
+\Delta C_{G}\left( p^{\ast },\Delta m\right) \text{, by (\ref%
{incrementcost_comparison})} \\
&\geq &-\Delta C_{G}\left( p,-\Delta m\right) \text{, by (\ref%
{increment_cost_inequality})} \\
&=&\Delta C_{\tilde{G}}\left( p,\Delta m\right) \text{, by (\ref{alternative}%
)} \\
&\geq &\left[ \left( \mathfrak{m}\left( \mathbf{b}_{s}\right) \right)
^{\alpha }-\left( \mathfrak{m}\left( \mathbf{b}_{s}\right) -\Delta m\right)
^{\alpha }\right] \int_{\gamma _{p}}ds\text{, by (\ref%
{incrementcost_comparison})} \\
&=&\left( \Delta m\right) ^{\alpha }\rho _{\alpha }\left( \mathfrak{m}\left(
\mathbf{b}_{s}\right) ,\Delta m\right) \int_{\gamma _{p}}ds\geq \left(
\Delta m\right) ^{\alpha }\rho _{\alpha }\left( \mathfrak{m}\left( \mathbf{b}%
_{s}\right) ,\Delta m\right) \left\Vert p-x_{s}\right\Vert \text{.}
\end{eqnarray*}

In particular, when $p^{\ast }=x_{i}$ for some $i$, (\ref{p_p*}) becomes (%
\ref{p_x_i}) as $\int_{\gamma _{p^{\ast }}}ds=0$.

Now, suppose $p^{\ast }$ is on $G_{i}$ with $i\neq s$. Apply (\ref{p_x_i})
to $p^{\ast }$, we have
\begin{equation*}
\frac{\left\Vert p^{\ast }-x_{i}\right\Vert }{\rho _{\alpha }\left(
\mathfrak{m}\left( \mathbf{b}_{i}\right) ,\Delta m\right) }\geq \int_{\gamma
_{p^{\ast }}}ds.
\end{equation*}%
Therefore,
\begin{eqnarray*}
&&\left\Vert p-p^{\ast }\right\Vert +\frac{\rho _{\alpha }\left( \theta
\left( p^{\ast }\right) +\Delta m,\Delta m\right) }{\rho _{\alpha }\left(
\mathfrak{m}\left( \mathbf{b}_{i}\right) ,\Delta m\right) }\left\Vert
p^{\ast }-x_{i}\right\Vert \\
&\geq &\left\Vert p-p^{\ast }\right\Vert +\rho _{\alpha }\left( \theta
\left( p^{\ast }\right) +\Delta m,\Delta m\right) \int_{\gamma _{p^{\ast
}}}ds \\
&\geq &\rho _{\alpha }\left( \mathfrak{m}\left( \mathbf{b}_{s}\right)
,\Delta m\right) \int_{\gamma _{p}}ds\text{, by (\ref{p_p*})} \\
&\geq &\rho _{\alpha }\left( \mathfrak{m}\left( \mathbf{b}_{s}\right)
,\Delta m\right) \left\Vert p-x_{s}\right\Vert \text{.}
\end{eqnarray*}
\end{proof}

Proposition \ref{prop key inequality} rests on the key inequality (\ref%
{increment_cost_inequality}), which follows intuitively by standard marginal
argument. For convenience of illustration, let $p$ denote the location $%
y_{j} $ of household $j$ who is connected to factory $s$ by some curve $%
\gamma _{p}.$ A planner will find it not optimal to choose an allocation
path $G$ such that $\Delta C_{G}\left( y_{j},-\Delta m\right) +\left( \Delta
m\right) ^{\alpha }\left\Vert y_{j}-p^{\ast }\right\Vert +\Delta C_{G}\left(
p^{\ast },\Delta m\right) <0$ for $p^{\ast }$ on some $G_{i}.$ It is because
in this case the planner has a less costly alternative by transferring $%
\Delta m$ amount of production from factory $s$ to factory $i$ and
transporting this additional $\Delta m$ units of commodity from factory $i$
first to a stopover point $p^{\ast }$ via curve $\gamma _{p^{\ast }}$ and
then directly from $p^{\ast }$ to household $j$. It's clear that this
strategy will send the same amount of commodity to household $j$ as before.
However, by the inequality, the reduction in transportation cost $-\Delta
C_{G}\left( p,-\Delta m\right) $ on curve $\gamma _{p}$ exceeds its increase
counterpart $\Delta C_{G}\left( p^{\ast },\Delta m\right) +\left( \Delta
m\right) ^{\alpha }\left\Vert y_{j}-p^{\ast }\right\Vert ,$ which cannot be
the case for an optimal allocation path.

By means of this proposition, we easily obtain the following results
regarding the properties of optimal assignment maps.

\begin{theorem}
\label{Theorem marginal}Suppose $S\in Map\left[ \ell ,k\right] $ is an
optimal assignment map. Let $j\in \left\{ 1,\cdots ,\ell \right\} $ and $%
s\in \left\{ 1,\cdots ,k\right\} $. If
\begin{equation}
y_{j}\in \digamma _{S}^{s}\left( n_{j}\right) :=\left\{ z\in \mathbb{R}%
^{m}:\min_{i\neq s}\left\Vert z-x_{i}\right\Vert <\rho _{\alpha }\left(
\mathfrak{m}\left( \mathbf{b}_{s}\right) ,n_{j}\right) \left\Vert
z-x_{s}\right\Vert \right\} ,  \label{min_far}
\end{equation}%
then $S\left( j\right) \neq s$. If
\begin{equation}
y_{j}\in \Omega _{S}^{s}\left( n_{j}\right) :=\left\{ z\in \mathbb{R}%
^{m}:\left\Vert z-x_{s}\right\Vert <\min_{i\neq s}\rho _{\alpha }\left(
\mathfrak{m}\left( \mathbf{b}_{i}\right) ,n_{j}\right) \left\Vert
z-x_{i}\right\Vert \right\} ,  \label{min_close}
\end{equation}%
then $S\left( j\right) =s$.
\end{theorem}

\begin{proof}
Assume $y_{j}\in \digamma _{S}^{s}\left( n_{j}\right) $ but $S\left(
j\right) =s$. Let $G=G_{S}$ be an allocation path as in (\ref{G_S}). By
Theorem \ref{Theorem 1}, $G\in Path\left[ \mathbf{x},\mathbf{b}\right] $ is
an optimal allocation path. Clearly, $y_{j}$ is on $G_{s}$ when $S\left(
j\right) =s\,$. Apply (\ref{p_x_i}) to $p=y_{j}$ and $\Delta m=n_{j}$, we
have
\begin{equation*}
\min_{i\in \left\{ 1,\cdots ,k\right\} }\left\Vert y_{j}-x_{i}\right\Vert
\geq \rho _{\alpha }\left( \mathfrak{m}\left( \mathbf{b}_{s}\right)
,n_{j}\right) \left\Vert y_{j}-x_{s}\right\Vert ,
\end{equation*}%
which contradicts (\ref{min_far}).

On the other hand, for each $i\neq s$, if $y_{j}\in \Omega _{S}^{s}\left(
n_{j}\right) $, then
\begin{equation*}
\min_{i^{\ast }\neq i}\left\Vert y_{j}-x_{i^{\ast }}\right\Vert \leq
\left\Vert y_{j}-x_{s}\right\Vert <\rho _{\alpha }\left( \mathfrak{m}\left(
\mathbf{b}_{i}\right) ,n_{j}\right) \left\Vert y_{j}-x_{i}\right\Vert .
\end{equation*}%
By (\ref{min_far}), $y_{j}\in \digamma _{S}^{i}\left( n_{j}\right) $ and $%
S\left( j\right) \neq i$ for such $i\neq s$. Thus, $S\left( j\right) =s$.
\end{proof}

Intuitively speaking, inequality (\ref{min_far}) says that if household $j$
locates \textquotedblleft closer\textquotedblright\ to some factory than
factory $s$, then the planner will not assign her to the factory $s$. Here
the relative closeness is weighted by a number $\rho _{\alpha }\left(
\mathfrak{m}\left( \mathbf{b}_{s}\right) ,n_{j}\right) $. When the
production $\mathfrak{m}\left( \mathbf{b}_{s}\right) $ at factory $s$ is
low, due to the transport economy of scale, one would expect that the
planner would less likely assign household $j$ to factory $s$. This
predication is justified by Theorem \ref{Theorem marginal} because in this
case, inequality (\ref{min_far}) becomes more likely to hold as $\rho
_{\alpha }\left( \mathfrak{m}\left( \mathbf{b}_{s}\right) ,n_{j}\right) $ is
high. The later part (\ref{min_close}) of the theorem states a special case
that if household $j$ is located uniformly closer to a factory $s$ than to
other factories, then she will be assigned to factory $s$ under any optimal
assignment map.

We now give a geometric description of the sets $\digamma _{S}^{s}\left(
n_{j}\right) $ and $\Omega _{S}^{s}\left( n_{j}\right) $.

\begin{lemma}
\label{lemma ball}For any constant $C\in \left( 0,1\right) $, the set
\begin{equation*}
\left\{ x\in \mathbb{R}^{m}:\left\Vert x-x_{i}\right\Vert <C\left\Vert
x-x_{s}\right\Vert \right\} =B\left( x_{i}+\frac{C^{2}}{1-C^{2}}\left(
x_{i}-x_{s}\right) ,\frac{C}{1-C^{2}}\left\Vert x_{i}-x_{s}\right\Vert
\right) ,
\end{equation*}%
where $B\left( x,r\right) $ denotes the open ball $\left\{ z\in \mathbb{R}%
^{m}:\left\Vert z-x\right\Vert <r\right\} $.
\end{lemma}

\begin{proof}
Indeed,
\begin{eqnarray*}
& &\left\Vert x-x_{i}\right\Vert <C\left\Vert x-x_{s}\right\Vert \\
&\iff &\left\Vert x-x_{i}\right\Vert ^{2}<C^{2}\left\Vert x-x_{s}\right\Vert
^{2} \\
&\iff &\left\Vert x\right\Vert ^{2}-2x\cdot x_{i}+\left\Vert
x_{i}\right\Vert ^{2}<C^{2}\left( \left\Vert x\right\Vert ^{2}-2x\cdot
x_{s}+\left\Vert x_{s}\right\Vert ^{2}\right) \\
&\iff &\left\Vert x\right\Vert ^{2}-2x\cdot \frac{x_{i}-C^{2}x_{s}}{1-C^{2}}<%
\frac{C^{2}\left\Vert x_{s}\right\Vert ^{2}-\left\Vert x_{i}\right\Vert ^{2}%
}{1-C^{2}} \\
&\iff &\left\Vert x-x_{i}-\frac{C^{2}}{1-C^{2}}\left( x_{i}-x_{s}\right)
\right\Vert ^{2}<\frac{C^{2}}{\left( 1-C^{2}\right) ^{2}}\left\Vert
x_{i}-x_{s}\right\Vert ^{2}.
\end{eqnarray*}
\end{proof}

Clearly, the ball $B\left( x_{i}+\frac{C^{2}}{1-C^{2}}\left(
x_{i}-x_{s}\right) ,\frac{C}{1-C^{2}}\left\Vert x_{i}-x_{s}\right\Vert
\right) $ given above contains $x_{i}$ but not $x_{s}$. By Lemma \ref{lemma
ball}, (\ref{min_far}) says that geometrically, if household $j$ lies in the
union of $k-1$ open balls%
\begin{equation*}
\digamma _{S}^{s}\left( n_{j}\right) =\bigcup_{i\neq s}B\left( x_{i}+\frac{%
\left( w_{sj}\right) ^{2}}{1-\left( w_{sj}\right) ^{2}}\left(
x_{i}-x_{s}\right) ,\frac{w_{sj}}{1-\left( w_{sj}\right) ^{2}}\left\Vert
x_{s}-x_{i}\right\Vert \right) ,
\end{equation*}%
where $w_{sj}=\rho _{\alpha }\left( \mathfrak{m}\left( \mathbf{b}_{s}\right)
,n_{j}\right) $, then $S\left( j\right) \neq s$. Also, (\ref{min_close})
says that if the household $j$ lies in the intersection of $\left(
k-1\right) $ balls
\begin{equation}
\Omega _{S}^{s}\left( n_{j}\right) =\bigcap_{i\neq s}B\left( x_{s}+\frac{%
\left( w_{ij}\right) ^{2}}{1-\left( w_{ij}\right) ^{2}}\left(
x_{s}-x_{i}\right) ,\frac{w_{ij}}{1-\left( w_{ij}\right) ^{2}}\left\Vert
x_{i}-x_{s}\right\Vert \right)  \label{Hj(ns)}
\end{equation}
for some $s$, then $S\left( j\right) =s$.

Since $\mathfrak{m}\left( \mathbf{b}_{s}\right) \leq \mathfrak{m}\left(
\mathbf{b}\right) =1$ and the function $\rho _{\alpha }\left( \cdot
,n_{j}\right) $ is decreasing, we have $\rho _{\alpha }\left( 1,n_{j}\right)
\leq \rho _{\alpha }\left( \mathfrak{m}\left( \mathbf{b}_{s}\right)
,n_{j}\right) $ and thus
\begin{eqnarray}
\digamma ^{s}\left( n_{j}\right) &:&=\left\{ z:\min_{i\neq s}\left\Vert
z-x_{i}\right\Vert <\rho _{\alpha }(1,n_{j})\left\Vert z-x_{s}\right\Vert
\right\} \subseteq \digamma _{S}^{s}\left( n_{j}\right) \text{,}
\label{F0(nj)} \\
\Omega ^{s}\left( n_{j}\right) &:&=\left\{ z:\left\Vert z-x_{s}\right\Vert
<\rho _{\alpha }\left( 1,n_{j}\right) \min_{s\neq i}\left\Vert
z-x_{i}\right\Vert \right\} \subseteq \Omega _{S}^{s}\left( n_{j}\right)
\label{Omega0(nj)}
\end{eqnarray}%
for any optimal assignment map $S$. Note that the set $\digamma ^{s}\left(
n_{j}\right) $ (or $\Omega ^{s}\left( n_{j}\right) $) is still the union (or
intersection) of $k-1$ balls in $\mathbb{R}^{m}$, and is independent of $S$.
For example, the sets $\Omega ^{s}\left( n_{j}\right) $ with $n_{j}=0.8$ and
$n_{j}=0.5$ are given by Figure \ref{G_2}, where $x_{1}=(0,0),x_{2}=(2,0)$
and $x_{3}=(1,2)$.

\begin{figure}[h]
\centering
\includegraphics[width=0.5\textwidth]{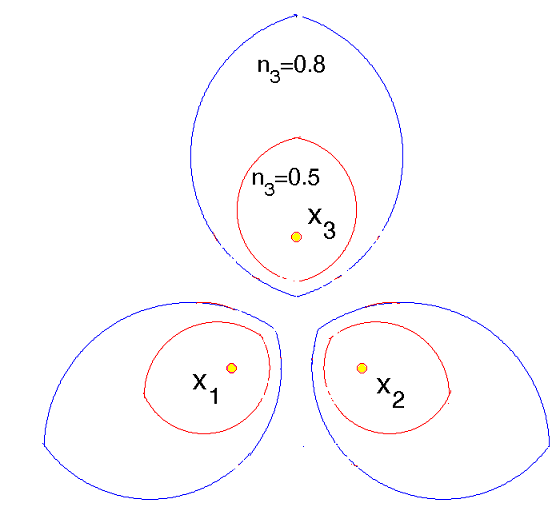}
\caption{An example of the sets $\Omega ^{s}\left( n_{j}\right) $ with $%
n_{j}=0.8$ (blue) and $n_{j}=0.5$ (red) when $\protect\alpha =1/2$.}
\label{G_2}
\end{figure}

\begin{corollary}
\label{Corollary Omega0nj}For any optimal assignment map $S\in Map\left[
\ell ,k\right] $ and $s\in \left\{ 1,\cdots ,k\right\} $, if $y_{j}\in
\digamma ^{s}\left( n_{j}\right) $, then $S\left( j\right) \neq s$. If $%
y_{j}\in \Omega ^{s}\left( n_{j}\right) $, then $S\left( j\right) =s$.
\end{corollary}

\begin{proof}
It follows from Theorem \ref{Theorem marginal}, (\ref{F0(nj)}) and (\ref%
{Omega0(nj)}).
\end{proof}

This corollary shows that if the household $j$ falls into the region $\Omega
^{s}\left( n_{j}\right) $ of some factory $s,$ then she will be assigned to
this factory under any optimal assignment map $S$. As a result, if all
households belong to the union of regions $\Omega ^{s}\left( n_{j}\right) $
of factories $s$ except factory $i$, then factory $i$ will not be used. Note
that as $\rho _{\alpha }\left( 1,n_{j}\right) $ is increasing in $n_{j},$
the size of the region $\Omega ^{s}\left( n_{j}\right) $ increases with $%
n_{j}$ as shown in Figure \ref{G_2}.

\begin{theorem}
\label{Theorem_neighborhood}Suppose $S\in Map\left[ \ell ,k\right] $ is an
optimal assignment map, $h$ and $j\in \left\{ 1,\cdots ,\ell \right\} $ with
$n_{j}\leq n_{h}$. If $S\left( h\right) =s^{\ast }\neq s$ for some $s\in
\left\{ 1,\cdots ,k\right\} $ and
\begin{equation}
y_{j}\in \digamma _{S}^{s,h}\left( n_{j}\right) :=\left\{ z\in \mathbb{R}%
^{m}\left\vert
\begin{array}{l}
\left\Vert z-y_{h}\right\Vert +\frac{\rho _{\alpha }\left(
n_{h}+n_{j},n_{j}\right) }{\rho _{\alpha }\left( \mathfrak{m}\left( \mathbf{b%
}_{s^{\ast }}\right) ,n_{j}\right) }\left\Vert y_{h}-x_{s^{\ast }}\right\Vert
\\
<\rho _{\alpha }\left( \mathfrak{m}\left( \mathbf{b}_{s}\right)
,n_{j}\right) \left\Vert z-x_{s}\right\Vert%
\end{array}%
\right\} \right. ,  \label{min_far_neigh}
\end{equation}%
then $S\left( j\right) \neq s$. If $S\left( h\right) =s$ for some $s\in
\left\{ 1,\cdots ,k\right\} $ and
\begin{equation}
y_{j}\in \Omega _{S}^{s,h}\left( n_{j}\right) :=\bigcap_{i\neq s,\mathfrak{m}%
\left( \mathbf{b}_{i}\right) \geq n_{j}}\left\{ z\in \mathbb{R}%
^{m}\left\vert
\begin{array}{l}
\left\Vert z-y_{h}\right\Vert +\frac{\rho _{\alpha }\left(
n_{h}+n_{j},n_{j}\right) }{\rho _{\alpha }\left( \mathfrak{m}\left( \mathbf{b%
}_{s}\right) ,n_{j}\right) }\left\Vert y_{h}-x_{s}\right\Vert \\
<\rho _{\alpha }\left( \mathfrak{m}\left( \mathbf{b}_{i}\right)
,n_{j}\right) \left\Vert z-x_{i}\right\Vert%
\end{array}%
\right\} \right. ,  \label{min_close_neigh}
\end{equation}%
then $S\left( j\right) =s$.
\end{theorem}

\begin{proof}
In the first scenario, assume $S\left( h\right) =s^{\ast }\neq s$, $y_{j}\in
\digamma _{S}^{s,h}\left( n_{j}\right) $ but $S\left( j\right) =s$. Then
\begin{eqnarray*}
&&\left\Vert y_{j}-y_{h}\right\Vert +\frac{\rho _{\alpha }\left(
n_{h}+n_{j},n_{j}\right) }{\rho _{\alpha }\left( \mathfrak{m}\left( \mathbf{b%
}_{s^{\ast }}\right) ,n_{j}\right) }\left\Vert y_{h}-x_{s^{\ast }}\right\Vert
\\
&\geq &\left\Vert y_{j}-y_{h}\right\Vert +\frac{\rho _{\alpha }\left( \theta
\left( y_{h}\right) +n_{j},n_{j}\right) }{\rho _{\alpha }\left( \mathfrak{m}%
\left( \mathbf{b}_{s^{\ast }}\right) ,n_{j}\right) }\left\Vert
y_{h}-x_{s^{\ast }}\right\Vert \text{ as }\theta \left( y_{h}\right) \geq
n_{h}\text{,} \\
&\geq &\rho _{\alpha }\left( \mathfrak{m}\left( \mathbf{b}_{s}\right)
,n_{j}\right) \left\Vert y_{j}-x_{s}\right\Vert \text{, by (\ref%
{p_neighborhood}),}
\end{eqnarray*}%
a contradiction with $y_{j}\in \digamma _{S}^{s,h}\left( n_{j}\right) $.
Thus, $S\left( j\right) \neq s.$

Now, in the second scenario, assume $S\left( h\right) =s$, $y_{j}\in \Omega
_{S}^{s,h}\left( n_{j}\right) $ but $S\left( j\right) =i^{\ast }$ for some $%
i^{\ast }\neq s$. Then,
\begin{eqnarray*}
&&\left\Vert y_{j}-y_{h}\right\Vert +\frac{\rho _{\alpha }\left(
n_{h}+n_{j},n_{j}\right) }{\rho _{\alpha }\left( \mathfrak{m}\left( \mathbf{b%
}_{s}\right) ,n_{j}\right) }\left\Vert y_{h}-x_{s}\right\Vert \\
&\geq &\left\Vert y_{j}-y_{h}\right\Vert +\frac{\rho _{\alpha }\left( \theta
\left( y_{h}\right) +n_{j},n_{j}\right) }{\rho _{\alpha }\left( \mathfrak{m}%
\left( \mathbf{b}_{s}\right) ,n_{j}\right) }\left\Vert
y_{h}-x_{s}\right\Vert \text{ as }\theta \left( y_{h}\right) \geq n_{h}\text{%
,} \\
&\geq &\rho _{\alpha }\left( \mathfrak{m}\left( \mathbf{b}_{i^{\ast
}}\right) ,n_{j}\right) \left\Vert y_{j}-x_{i^{\ast }}\right\Vert \text{, by
(\ref{p_neighborhood}),}
\end{eqnarray*}%
a contradiction with $y_{j}\in \Omega _{S}^{s,h}\left( n_{j}\right) $ as $%
\mathfrak{m}\left( \mathbf{b}_{i^{\ast }}\right) \geq n_{j}$ and $i^{\ast
}\neq s$. Thus, $S\left( j\right) =s.$
\end{proof}

The first part of Theorem \ref{Theorem_neighborhood} says that if some
household $h$ is not assigned to factory $s$, then any other nearby
household $j$ (i.e. within the neighborhood region $\digamma
_{S}^{s,h}\left( n_{j}\right) $) with a smaller demand will also not be
assigned to factory $s$. The second part says that if some household $h$ is
assigned to factory $s$, then any other nearby household $j$ (i.e. within
the neighborhood region $\Omega _{S}^{s,h}\left( n_{j}\right) $) with a
smaller demand will also be assigned to factory $s$. These findings agree
with the intuition that grouping with nearby households of large demand
would make it more likely to realize the benefit of transport economy of
scale.

\section{Properties of Optimal Assignment Maps via Projectional Analysis}

As seen in the previous section, under an optimal assignment map, a
household will be assigned to some factory if she lives close to the factory
(Theorem \ref{Theorem marginal}) or she has some nearby neighbors assigned
to the factory (Theorem \ref{Theorem_neighborhood}). In this section, we
will show a reverse result (Theorem \ref{autarky}) using a method of
projectional analysis.

Throughout this section, we consider the projection map from $\mathbb{R}^{m}$
to the line $\left\{ p+tv:t\in \mathbb{R}\right\} $ for fixed points $p,v\in
\mathbb{R}^{m}$ with $\left\Vert v\right\Vert =1$. Under this map, each
point $z\in \mathbb{R}^{m}$ is mapped to $p+\pi \left( z\right) v$ with
\begin{equation}
\pi \left( z\right) =\left\langle z-p,v\right\rangle ,  \label{pi_projection}
\end{equation}%
where $\left\langle \cdot ,\cdot \right\rangle $ stands for the standard
inner product in $\mathbb{R}^{m}$. For instance, when $p=\left( 0,\cdots
,0\right) \in \mathbb{R}^{m}$, and $v=\left( 1,0,\cdots ,0\right) \in
\mathbb{R}^{m}$, for each $z=\left( z_{1},\cdots ,z_{m}\right) $, $\pi
\left( z\right) =z_{1}$ gives the first coordinate of $z$.

We start with two lemmas regarding properties of a single-source transport
system. These lemmas will play a crucial role in establishing Theorem \ref%
{autarky} later.

\begin{lemma}
\label{Comparison} Suppose
\begin{equation}
\mathbf{c=}\sum_{j\in \Theta }c_{j}\delta _{z_{j}}\text{, }c_{j}>0\text{ for
}j\text{ in a finite set }\Theta \text{ }  \label{measure_c}
\end{equation}%
is an atomic measure on $X\subseteq \mathbb{R}^{m}$, and $P,Q\in X$. If
there exists a $t_{1}$ such that
\begin{equation*}
\pi \left( P\right) \geq t_{1}\geq \max_{j\in \Theta }\pi \left(
z_{j}\right) \text{ }\ \text{or }\pi \left( P\right) \leq t_{1}\leq
\min_{j\in \Theta }\pi \left( z_{j}\right) \text{,}
\end{equation*}%
then
\begin{equation*}
d_{\alpha }\left( \mathbf{c},\mathfrak{m}\left( \mathbf{c}\right) \delta
_{P}\right) -d_{\alpha }\left( \mathbf{c},\mathfrak{m}\left( \mathbf{c}%
\right) \delta _{Q}\right) \geq \mathfrak{m}\left( \mathbf{c}\right)
^{\alpha }\left( |\pi \left( P\right) -t_{1}|-CR-|\pi \left( Q\right)
-t_{1}|\right) ,
\end{equation*}%
where
\begin{equation}
C=\frac{\sqrt{m-1}}{2^{1-\left( m-1\right) \left( 1-\alpha \right) }-1}+1,
\label{ConstantC}
\end{equation}%
and
\begin{equation}
R=\max \left\{ \left\Vert z-p-\pi \left( z\right) v\right\Vert :z\in \left\{
z_{j}:j\in \Theta \right\} \cup \left\{ P,Q\right\} \right\} .
\label{constantR}
\end{equation}
\end{lemma}

\begin{proof}
Without loss of generality, we assume that $\pi \left( P\right) \geq
t_{1}\geq \max_{j\in \Theta }\pi \left( z_{j}\right) $. Let $G\in Path\left(
\mathbf{c},\mathfrak{m}\left( \mathbf{c}\right) \delta _{P}\right) $ be an
optimal transport path. As in (\ref{g_matrix}), there exists a unique curve $%
\gamma _{z_{j}}$ from $P$ to $z_{j}$ for each $j$. Since $\pi \left(
P\right) \geq t_{1}\geq $ $\pi \left( z_{j}\right) $, there exists a point $%
v_{j}$ on the curve $\gamma _{z_{j}}$ such that $\pi \left( v_{j}\right)
=t_{1}$. Let
\begin{equation*}
\mathbf{\tilde{c}=}\sum_{j\in \Theta }c_{j}\delta _{v_{j}},
\end{equation*}%
then, by the optimality of $G$, we have%
\begin{equation}
d_{\alpha }\left( \mathbf{c},\mathfrak{m}\left( \mathbf{c}\right) \delta
_{P}\right) =d_{\alpha }\left( \mathbf{c},\mathbf{\tilde{c}}\right)
+d_{\alpha }\left( \mathbf{\tilde{c}},\mathfrak{m}\left( \mathbf{c}\right)
\delta _{P}\right) .  \label{d_alpha_cut}
\end{equation}%
For $O_{1}=p+t_{1}v$ and $R$ given in (\ref{constantR}), we consider the set
\begin{equation*}
B_{R}\left( O_{1}\right) :=\left\{ z\in \mathbb{R}^{m}:\pi \left( z\right)
=t_{1}\text{, }\left\Vert z-O_{1}\right\Vert \leq R\right\} ,
\end{equation*}%
which is an $\left( m-1\right) $ dimensional ball perpendicular to the line
passing through $p$ in the direction $v$. By means of (\ref{constantR}), the
measure $\mathbf{\tilde{c}}$ is supported on $B_{R}\left( O_{1}\right) $.
Let
\begin{equation*}
\tilde{P}=P+\left( t_{1}-\pi \left( P\right) \right) v
\end{equation*}%
be the projection point of $P$ onto the ball $B_{R}\left( O_{1}\right) $.
One can show as in Xia and Vershynina \cite[Lemma 2.3.1]{xia9} that
\begin{equation*}
d_{\alpha }\left( \mathbf{\tilde{c}},\mathfrak{m}\left( \mathbf{c}\right)
\delta _{P}\right) \geq d_{\alpha }\left( \mathfrak{m}\left( \mathbf{c}%
\right) \delta _{\tilde{P}},\mathfrak{m}\left( \mathbf{c}\right) \delta
_{P}\right) =\mathfrak{m}\left( \mathbf{c}\right) ^{\alpha }|\pi \left(
P\right) -t_{1}|.
\end{equation*}%
On the other hand, by Xia \cite[Theorem 3.1]{xia1}, we have
\begin{eqnarray*}
d_{\alpha }\left( \mathbf{\tilde{c}},\mathfrak{m}\left( \mathbf{c}\right)
\delta _{Q}\right) &\leq &d_{\alpha }\left( \mathbf{\tilde{c}},\mathfrak{m}%
\left( \mathbf{c}\right) \delta _{O_{1}}\right) +d_{\alpha }\left( \mathfrak{%
m}\left( \mathbf{c}\right) \delta _{O_{1}},\mathfrak{m}\left( \mathbf{c}%
\right) \delta _{Q}\right) \\
&\leq &\mathfrak{m}\left( \mathbf{c}\right) ^{\alpha }\tilde{C}R+\mathfrak{m}%
\left( \mathbf{c}\right) ^{\alpha }\left( R+|\pi \left( Q\right)
-t_{1}|\right) ,
\end{eqnarray*}%
where
\begin{equation*}
\tilde{C}=\frac{\sqrt{m-1}}{2^{1-\left( m-1\right) \left( 1-\alpha \right)
}-1}\text{.}
\end{equation*}%
Therefore,
\begin{eqnarray*}
&&d_{\alpha }\left( \mathbf{c},\mathfrak{m}\left( \mathbf{c}\right) \delta
_{P}\right) -d_{\alpha }\left( \mathbf{c},\mathfrak{m}\left( \mathbf{c}%
\right) \delta _{Q}\right) \\
&\geq &d_{\alpha }\left( \mathbf{c},\mathfrak{m}\left( \mathbf{c}\right)
\delta _{P}\right) -\left[ d_{\alpha }\left( \mathbf{c},\mathbf{\tilde{c}}%
\right) +d_{\alpha }\left( \mathbf{\tilde{c}},\mathfrak{m}\left( \mathbf{c}%
\right) \delta _{Q}\right) \right] \text{, by the triangle inequality} \\
&=&d_{\alpha }\left( \mathbf{\tilde{c}},\mathfrak{m}\left( \mathbf{c}\right)
\delta _{P}\right) -d_{\alpha }\left( \mathbf{\tilde{c}},\mathfrak{m}\left(
\mathbf{c}\right) \delta _{Q}\right) \text{, by (\ref{d_alpha_cut})} \\
&\geq &\mathfrak{m}\left( \mathbf{c}\right) ^{\alpha }|\pi \left( P\right)
-t_{1}|-\mathfrak{m}\left( \mathbf{c}\right) ^{\alpha }\left( CR+|\pi \left(
Q\right) -t_{1}|\right) .
\end{eqnarray*}
\end{proof}

\begin{lemma}
\label{decomposition}Let $\mathbf{c}$ be an atomic measure as given in (\ref%
{measure_c}) and $O\in X$. If the set $\Theta $ is decomposed as the
disjoint union of two nonempty subsets
\begin{equation*}
\Theta =\Theta _{1}\amalg \Theta _{2},
\end{equation*}%
then, for any optimal transport path $G\in Path\left( \mathfrak{m}\left(
\mathbf{c}\right) \delta _{O},\mathbf{c}\right) $, there exist a vertex
point $P\in V\left( G\right) $ and a decomposition of each $\Theta _{i}$:
\begin{equation*}
\Theta _{i}=\tilde{\Theta}_{i}\amalg \bar{\Theta}_{i}\text{\textbf{\ }with }%
\tilde{\Theta}_{i}\text{ nonempty, }i=1,2
\end{equation*}%
such that $G$ can be decomposed as%
\begin{equation}
G=G_{1}+G_{2}+G_{3}  \label{G_decompoistion}
\end{equation}%
where for $i=1,2$, $G_{i}$ is an optimal transport path from $\mathfrak{m}%
\left( \mathbf{\tilde{c}}_{i}\right) \delta _{P}$ to $\mathbf{\tilde{c}}_{i}$
for
\begin{equation}
\mathbf{\tilde{c}}_{i}=\sum_{j\in \tilde{\Theta}_{i}}c_{j}\delta _{z_{j}}%
\text{ ,}  \label{equation_c_tilde}
\end{equation}%
$G_{3}$ is an optimal transport path from $\mathfrak{m}\left( \mathbf{c}%
\right) \delta _{O}$ to $\mathbf{\bar{c}+}\left( \mathfrak{m}\left( \mathbf{%
\tilde{c}}_{1}\right) +\mathfrak{m}\left( \mathbf{\tilde{c}}_{2}\right)
\right) \delta _{P}$ for
\begin{equation*}
\mathbf{\bar{c}}=\sum_{j\in \bar{\Theta}_{1}\cup \bar{\Theta}%
_{2}}c_{j}\delta _{z_{j}},
\end{equation*}%
and $\left\{ G_{i}\right\} _{i=1}^{3}$ are pairwise disjoint except at $P$.
Moreover, (\ref{G_decompoistion}) implies
\begin{equation}
\mathbf{M}_{\alpha }\left( G\right) =\mathbf{M}_{\alpha }\left( G_{1}\right)
+\mathbf{M}_{\alpha }\left( G_{2}\right) +\mathbf{M}_{\alpha }\left(
G_{3}\right)  \label{M_G_decomposition}
\end{equation}%
and by the optimality of $G$, it follows%
\begin{eqnarray}
d_{\alpha }\left( \mathfrak{m}\left( \mathbf{c}\right) \delta _{O},\mathbf{c}%
\right) &=&d_{\alpha }\left( \mathfrak{m}\left( \mathbf{\tilde{c}}%
_{1}\right) \delta _{P},\mathbf{\tilde{c}}_{1}\right) +d_{\alpha }\left(
\mathfrak{m}\left( \mathbf{\tilde{c}}_{2}\right) \delta _{P},\mathbf{\tilde{c%
}}_{2}\right)  \label{d_a_sum} \\
&&+d_{\alpha }\left( \mathfrak{m}\left( \mathbf{c}\right) \delta _{O},%
\mathbf{\bar{c}+}\left( \mathfrak{m}\left( \mathbf{\tilde{c}}_{1}\right) +%
\mathfrak{m}\left( \mathbf{\tilde{c}}_{2}\right) \right) \delta _{P}\right) .
\notag
\end{eqnarray}
\end{lemma}

\begin{proof}
For any $z$ on the support of $G$, since $G$ is a transport path from a
single source $O$, there exists a unique curve $\gamma _{z}$ on $G$ from $O$
to $z$. Also, it is easily observed that
\begin{equation}
\text{\textit{if }}\tilde{z}\text{\textit{\ lies on }}\gamma _{z}\text{%
\textit{\ for some }}z\text{\textit{, then }}\gamma _{\tilde{z}}\text{%
\textit{\ is the part of }}\gamma _{z}\text{\textit{\ from }}O\text{\textit{%
\ to }}\tilde{z}.  \label{single_source}
\end{equation}

Now, let $\Gamma _{i}$ be the union of all curves $\gamma _{z_{j}}$ with $%
j\in \Theta _{i}$ for $i=1,2$, and set
\begin{equation*}
\Gamma =\Gamma _{1}\cap \Gamma _{2}.
\end{equation*}%
By (\ref{single_source}), if $z\in \Gamma $, then $\gamma _{z}\subseteq
\Gamma $. This shows that $\Gamma $ is a connected subset of the support of $%
G$ containing $O$. Since $\Gamma $ contains no cycles, it is a contractible
set containing $O$. Then, by calculating the Euler characteristic number of $%
\Gamma $, we have either $\Gamma =\left\{ O\right\} $ or $\Gamma $ has at
least two endpoints (i.e. vertices of degree $1$).

If $\Gamma =\left\{ O\right\} $, then set $P=O$, and $\tilde{\Theta}%
_{i}=\Theta _{i}$ for $i=1,2$. If $\Gamma \neq \left\{ O\right\} $, pick $P$
to be an endpoint of \ $\Gamma $ with $P\neq O$. Since $P\in \Gamma
\subseteq \Gamma _{i}$, the set
\begin{equation*}
\tilde{\Theta}_{i}:=\left\{ j\in \Theta _{i}:P\in \gamma _{z_{j}}\right\}
\neq \emptyset ,\text{ for }i=1,2.
\end{equation*}

For any $j\in \tilde{\Theta}_{i}$, $P$ divides the curve $\gamma _{z_{j}}$
into two parts: $\gamma _{z_{j}}^{\left( 1\right) }$ from $O$ to $P$ and $%
\gamma _{z_{j}}^{\left( 2\right) }$ from $P$ to $z_{j}$. Since $P$ is an
endpoint of $\Gamma $, we have
\begin{equation*}
\left( \gamma _{z_{j}}^{\left( 2\right) }\backslash \left\{ P\right\}
\right) \cap \Gamma =\emptyset \text{.}
\end{equation*}%
For $i=1,2$, define $\mathbf{\tilde{c}}_{i}$ using (\ref{equation_c_tilde})
and denote the part of $G$ from $\mathfrak{m}\left( \mathbf{\tilde{c}}%
_{i}\right) \delta _{P}$ to $\mathbf{\tilde{c}}_{i}$ by $G_{i}$. The rest of
$G$ is denoted by $G_{3}=G-\left( G_{1}+G_{2}\right) $. Then, by
construction, $\left\{ G_{i}\right\} _{i=1}^{3}$ are pairwise disjoint
except at $P$, and thus (\ref{M_G_decomposition}) holds. By the optimality
of $G$, each $G_{i}$ must also be optimal for $i=1,2,3$, which yields (\ref%
{d_a_sum}).
\end{proof}

The following theorem states that: under an optimal assignment map, a
household will be assigned to some factory only when either she lives close
to the factory or she has some nearby neighbors assigned to the factory. In
the first situation, the planner takes advantage of relative spatial
locations between households and factories; while in the second situation,
the planner takes advantage of group transportation due to transport economy
of scale embedded in ramified transport technology.

Let $S\in Map\left[ \ell ,k\right] $ be an optimal assignment map. For each $%
i\in \left\{ 1,\cdots ,k\right\} $, define%
\begin{equation}
\Psi _{i}=\left\{ y_{j}:S\left( j\right) =i\right\} \cup \left\{
x_{i}\right\}  \label{Psi}
\end{equation}%
and%
\begin{equation}
R_{i}=\max \left\{ \left\Vert z-p-\pi \left( z\right) v\right\Vert :z\in
\left\{ y_{j}:S\left( j\right) =i\right\} \cup \left\{ x_{1},\cdots
,x_{k}\right\} \right\} .  \label{R_i}
\end{equation}

\begin{theorem}
\label{autarky} Suppose $S\in Map\left[ \ell ,k\right] $ is an optimal
assignment map. Then, for any $i\in \left\{ 1,\cdots ,k\right\} $ and $j\in
\left\{ 1,\cdots ,\ell \right\} $ with $S\left( j\right) =i$, there exists $%
z\in \Psi _{i}\backslash \left\{ y_{j}\right\} $, such that
\begin{equation}
0<\left\vert \pi \left( y_{j}\right) -\pi \left( z\right) \right\vert \leq
2CR_{i}+\min_{i^{\ast }\neq i}\left\vert \pi \left( x_{i^{\ast }}\right)
-\pi \left( y_{j}\right) \right\vert  \label{projection_closeness}
\end{equation}%
and $\pi \left( z\right) $ is between $\pi \left( y_{j}\right) $ and $\pi
\left( x_{i}\right) $, where $C$ and $R_{i}$ are the constants given in (\ref%
{ConstantC}) and (\ref{R_i}) respectively.
\end{theorem}

\begin{proof}
Without loss of generality, we may assume that $\pi \left( y_{j}\right) \leq
\pi \left( x_{i}\right) $. Let
\begin{equation*}
\left\vert \pi \left( x_{i^{\ast }}\right) -\pi \left( y_{j}\right)
\right\vert =\min_{i^{\ast \ast }\neq i}\left\vert \pi \left( x_{i^{\ast
\ast }}\right) -\pi \left( y_{j}\right) \right\vert
\end{equation*}%
for some $i^{\ast }\neq i$. We want to prove (\ref{projection_closeness}) by
contradiction. Assume for any $z\in \Psi _{i}\backslash \left\{
y_{j}\right\} $ with $\pi \left( y_{j}\right) <\pi \left( z\right) \leq \pi
\left( x_{i}\right) $,%
\begin{equation*}
\left\vert \pi \left( y_{j}\right) -\pi \left( z\right) \right\vert
>2CR_{i}+\left\vert \pi \left( x_{i^{\ast }}\right) -\pi \left( y_{j}\right)
\right\vert ,
\end{equation*}%
i.e.%
\begin{equation*}
\pi \left( z\right) -\pi \left( y_{j}\right) >2CR_{i}+\left\vert \pi \left(
x_{i^{\ast }}\right) -\pi \left( y_{j}\right) \right\vert .
\end{equation*}%
Then, there exists a real number $t_{2}$ such that
\begin{equation}
\pi \left( z\right) >t_{2}>\pi \left( y_{j}\right) +2CR_{i}+\left\vert \pi
\left( x_{i^{\ast }}\right) -\pi \left( y_{j}\right) \right\vert
\label{t_2_condition}
\end{equation}%
whenever $\pi \left( y_{j}\right) <\pi \left( z\right) \leq \pi \left(
x_{i}\right) $. In particular,
\begin{equation}
\pi \left( x_{i}\right) >t_{2}>\pi \left( y_{j}\right) +2CR_{i}+\left\vert
\pi \left( x_{i^{\ast }}\right) -\pi \left( y_{j}\right) \right\vert
\label{x_i_condition}
\end{equation}%
as $x_{i}\in \Psi _{i}$. As a result, $S^{-1}\left( i\right) $ can be
expressed as the disjoint union of two sets:%
\begin{equation*}
S^{-1}\left( i\right) =\Theta _{1}\amalg \Theta _{2},
\end{equation*}%
where
\begin{equation}
\Theta _{1}:=\left\{ h\in S^{-1}\left( i\right) :\pi \left( y_{h}\right)
\leq \pi \left( y_{j}\right) \right\} \text{, }\Theta _{2}:=\left\{ h\in
S^{-1}\left( i\right) :\text{ }\pi \left( y_{h}\right) >t_{2}\right\} .
\label{Sigma12}
\end{equation}%
Clearly, $j\in \Theta _{1}$. If $\Theta _{2}=\emptyset $, then $S^{-1}\left(
i\right) =\Theta _{1}$ and thus
\begin{equation*}
\pi \left( x_{i}\right) \geq \pi \left( y_{j}\right) \geq \max_{h\in
S^{-1}\left( i\right) }\pi \left( y_{h}\right) .
\end{equation*}%
Let $\mathbf{a}_{i}$ and $\mathbf{b}_{i}$ be given as in (\ref{a_b_i}). By
Lemma $\ref{Comparison}$ with $t_{1}=\pi \left( y_{j}\right) $, we have
\begin{eqnarray*}
&&d_{\alpha }\left( \mathbf{b}_{i},\mathfrak{m}\left( \mathbf{b}_{i}\right)
\delta _{x_{i}}\right) -d_{\alpha }\left( \mathbf{b}_{i},\mathfrak{m}\left(
\mathbf{b}_{i}\right) \delta _{x_{i^{\ast }}}\right) \\
&\geq &\mathfrak{m}\left( \mathbf{b}_{i}\right) ^{\alpha }\left( |\pi \left(
x_{i}\right) -\pi \left( y_{j}\right) |-CR_{i}-|\pi \left( x_{i^{\ast
}}\right) -\pi \left( y_{j}\right) |\right) \\
&\geq &\mathfrak{m}\left( \mathbf{b}_{i}\right) ^{\alpha }\left( \pi \left(
x_{i}\right) -\pi \left( y_{j}\right) -CR_{i}-|\pi \left( x_{i^{\ast
}}\right) -\pi \left( y_{j}\right) |\right) \\
&>&\mathfrak{m}\left( \mathbf{b}_{i}\right) ^{\alpha }CR_{i}>0,\text{ by (%
\ref{x_i_condition}),}
\end{eqnarray*}%
a contradiction to the optimality of $S$. Thus, $\Theta _{2}\neq \emptyset $.

Let $G_{i}\in Path\left( \mathbf{a}_{i},\mathbf{b}_{i}\right) $ be an
optimal transport path. Then, by Theorem \ref{Theorem 1} and optimality of $%
S $, $G=\sum_{i}G_{i}$ is an optimal allocation path. Since both $\Theta _{1}
$ and $\Theta _{2}$ are nonempty, by setting $\Theta =S^{-1}\left( i\right)
\,=\Theta _{1}\amalg \Theta _{2},O=x_{i}$ and $\mathbf{c=b}_{i}$ in Lemma %
\ref{decomposition}, there exists a point $P\in V\left( G_{i}\right) $ such
that $G_{i}$ can be decomposed as%
\begin{equation*}
G_{i}=G_{i}^{\left( 1\right) }+G_{i}^{\left( 2\right) }+G_{i}^{\left(
3\right) }
\end{equation*}%
with%
\begin{equation}
\mathbf{M}_{\alpha }\left( G_{i}\right) =\mathbf{M}_{\alpha }\left(
G_{i}^{\left( 1\right) }\right) +\mathbf{M}_{\alpha }\left( G_{i}^{\left(
2\right) }\right) +\mathbf{M}_{\alpha }\left( G_{i}^{\left( 3\right)
}\right) .  \label{equation_M_G_i}
\end{equation}%
Here, for $h=1,2$, $G_{i}^{\left( h\right) }$ is an optimal transport path
from $\mathfrak{m}\left( \mathbf{\tilde{b}}_{i}^{\left( h\right) }\right)
\delta _{P}$ to $\mathbf{\tilde{b}}_{i}^{\left( h\right) }$ for some
positive atomic measures $\mathbf{\tilde{b}}_{i}^{\left( h\right) }$ with $%
spt\left( \mathbf{\tilde{b}}_{i}^{\left( h\right) }\right) \subseteq \left\{
y_{\tilde{h}}:\tilde{h}\in \Theta _{h}\right\} $ and
\begin{equation*}
G_{i}^{\left( 3\right) }\in Path\left( \mathfrak{m}\left( \mathbf{b}%
_{i}\right) \delta _{x_{i}},\mathbf{b}_{i}-\left( \mathbf{\tilde{b}}%
_{i}^{\left( 1\right) }+\mathbf{\tilde{b}}_{i}^{\left( 2\right) }\right)
\mathbf{+}\left( \mathfrak{m}\left( \mathbf{\tilde{b}}_{i}^{\left( 1\right)
}\right) +\mathfrak{m}\left( \mathbf{\tilde{b}}_{i}^{\left( 2\right)
}\right) \right) \delta _{P}\right) .
\end{equation*}

If $\pi \left( P\right) \geq t_{2}-CR_{i}$, then we can modify $G$ into
another allocation path $\tilde{G}$ by just replacing the corresponding
transport path from factory $i$ to households $\mathbf{\tilde{b}}%
_{i}^{\left( 1\right) }$ with an optimal transport path from factory $%
i^{\ast }$ to $\mathbf{\tilde{b}}_{i}^{\left( 1\right) }$. More precisely,
we replace $G_{i}$ by
\begin{equation}
\tilde{G}_{i}=\tilde{G}_{i}^{\left( 1\right) }+\tilde{G}_{i}^{\left(
2\right) }+\tilde{G}_{i}^{\left( 3\right) },  \label{G_bar}
\end{equation}%
where $\tilde{G}_{i}^{\left( 1\right) }$ is an optimal transport path from $%
\mathfrak{m}\left( \mathbf{\tilde{b}}_{i}^{\left( 1\right) }\right) \delta
_{x_{i^{\ast }}}$ to $\mathbf{\tilde{b}}_{i}^{\left( 1\right) }$, $\tilde{G}%
_{i}^{\left( 2\right) }=G_{i}^{\left( 2\right) }$ and
\begin{equation}
\tilde{G}_{i}^{\left( 3\right) }=G_{i}^{\left( 3\right) }-\mathfrak{m}\left(
\mathbf{\tilde{b}}_{i}^{\left( 1\right) }\right) \gamma _{P},  \label{G3}
\end{equation}%
where $\gamma _{P}$ is the curve on $G$ from $x_{i}$ to $P$. Equation (\ref%
{G_bar}) and (\ref{G3}) imply respectively
\begin{equation*}
\mathbf{M}_{\alpha }\left( \tilde{G}_{i}\right) \leq \mathbf{M}_{\alpha
}\left( \tilde{G}_{i}^{\left( 1\right) }\right) +\mathbf{M}_{\alpha }\left(
G_{i}^{\left( 2\right) }\right) +\mathbf{M}_{\alpha }\left( \tilde{G}%
_{i}^{\left( 3\right) }\right)
\end{equation*}%
and
\begin{equation*}
\mathbf{M}_{\alpha }\left( \tilde{G}_{i}^{\left( 3\right) }\right) \leq
\mathbf{M}_{\alpha }\left( G_{i}^{\left( 3\right) }\right) .
\end{equation*}%
Consequently, by (\ref{equation_M_G_i}),
\begin{eqnarray}
&&\mathbf{M}_{\alpha }\left( G_{i}\right) -\mathbf{M}_{\alpha }\left( \tilde{%
G}_{i}\right)  \notag \\
&\geq &\mathbf{M}_{\alpha }\left( G_{i}^{\left( 1\right) }\right) -\mathbf{M}%
_{\alpha }\left( \tilde{G}_{i}^{\left( 1\right) }\right) +\left( \mathbf{M}%
_{\alpha }\left( G_{i}^{\left( 3\right) }\right) -\mathbf{M}_{\alpha }\left(
\tilde{G}_{i}^{\left( 3\right) }\right) \right)  \notag \\
&\geq &\mathbf{M}_{\alpha }\left( G_{i}^{\left( 1\right) }\right) -\mathbf{M}%
_{\alpha }\left( \tilde{G}_{i}^{\left( 1\right) }\right)  \notag \\
&=&d_{\alpha }\left( \mathfrak{m}\left( \mathbf{\tilde{b}}_{i}^{\left(
1\right) }\right) \delta _{P},\mathbf{\tilde{b}}_{i}^{\left( 1\right)
}\right) -d_{\alpha }\left( \mathfrak{m}\left( \mathbf{\tilde{b}}%
_{i}^{\left( 1\right) }\right) \delta _{x_{i^{\ast }}},\mathbf{\tilde{b}}%
_{i}^{\left( 1\right) }\right) ,  \label{Comparison_G_i}
\end{eqnarray}%
where the last equality follows from the optimality of both $G_{i}^{\left(
1\right) }$ and $\tilde{G}_{i}^{\left( 1\right) }$.

Since $\alpha <1$, for $\tilde{G}=\sum_{s\neq i}G_{s}+\tilde{G}_{i}$, we have%
\begin{equation*}
\mathbf{M}_{\alpha }\left( \tilde{G}\right) \leq \sum_{s\neq i}\mathbf{M}%
_{\alpha }\left( G_{s}\right) +\mathbf{M}_{\alpha }\left( \tilde{G}%
_{i}\right) .
\end{equation*}%
Due to the optimality of $G$, equation (\ref{equation(M_G)}) says%
\begin{equation*}
\mathbf{M}_{\alpha }\left( G\right) =\sum_{s\neq i}\mathbf{M}_{\alpha
}\left( G_{s}\right) +\mathbf{M}_{\alpha }\left( G_{i}\right) .
\end{equation*}%
As a result,%
\begin{eqnarray*}
&&\mathbf{M}_{\alpha }\left( G\right) -\mathbf{M}_{\alpha }\left( \tilde{G}%
\right) \\
&\geq &\mathbf{M}_{\alpha }\left( G_{i}\right) -\mathbf{M}_{\alpha }\left(
\tilde{G}_{i}\right) \\
&\geq &d_{\alpha }\left( \mathfrak{m}\left( \mathbf{\tilde{b}}_{i}^{\left(
1\right) }\right) \delta _{P},\mathbf{\tilde{b}}_{i}^{\left( 1\right)
}\right) -d_{\alpha }\left( \mathfrak{m}\left( \mathbf{\tilde{b}}%
_{i}^{\left( 1\right) }\right) \delta _{x_{i^{\ast }}},\mathbf{\tilde{b}}%
_{i}^{\left( 1\right) }\right) \text{, by (\ref{Comparison_G_i})} \\
&\geq &\mathfrak{m}\left( \mathbf{\tilde{b}}_{i}^{\left( 1\right) }\right)
^{\alpha }\left( |\pi \left( P\right) -\pi \left( y_{j}\right) |-CR_{i}-|\pi
\left( x_{i^{\ast }}\right) -\pi \left( y_{j}\right) |\right) \text{, by
Lemma }\ref{Comparison}\text{ } \\
&\geq &\mathfrak{m}\left( \mathbf{\tilde{b}}_{i}^{\left( 1\right) }\right)
^{\alpha }\left( \pi \left( P\right) -\pi \left( y_{j}\right) -CR_{i}-|\pi
\left( x_{i^{\ast }}\right) -\pi \left( y_{j}\right) |\right) \\
&\geq &\mathfrak{m}\left( \mathbf{\tilde{b}}_{i}^{\left( 1\right) }\right)
^{\alpha }\left( t_{2}-\pi \left( y_{j}\right) -2CR_{i}-|\pi \left(
x_{i^{\ast }}\right) -\pi \left( y_{j}\right) |\right) >0\text{, by (\ref%
{t_2_condition}).}
\end{eqnarray*}%
Thus, $\mathbf{M}_{\alpha }\left( G\right) >\mathbf{M}_{\alpha }\left(
\tilde{G}\right) $, which contradicts the optimality of $G$, and thus the
inequality (\ref{projection_closeness}) must hold.

If $\pi \left( P\right) <t_{2}-CR_{i}$, then let $Q$ be the first point of $%
\gamma _{P}$ with $\pi \left( Q\right) =t_{2}$. We can modify $G$ into
another allocation path $\bar{G}$ by just replacing the corresponding
transport path from the point $Q$ to households $\mathbf{\tilde{b}}%
_{i}^{\left( 2\right) }$ with an optimal transport path from $Q$ to $\mathbf{%
\tilde{b}}_{i}^{\left( 2\right) }$. More precisely, we replace $G_{i}$ by
\begin{equation*}
\bar{G}_{i}=\bar{G}_{i}^{\left( 1\right) }+\bar{G}_{i}^{\left( 2\right) }+%
\bar{G}_{i}^{\left( 3\right) },
\end{equation*}%
where $\bar{G}_{i}^{\left( 1\right) }=G_{i}^{\left( 1\right) }$, $\tilde{G}%
_{i}^{\left( 2\right) }$ is an optimal transport path from $\mathfrak{m}%
\left( \mathbf{\tilde{b}}_{i}^{\left( 2\right) }\right) \delta _{Q}$ to $%
\mathbf{\tilde{b}}_{i}^{\left( 2\right) }$, and
\begin{equation*}
\tilde{G}_{i}^{\left( 3\right) }=G_{i}^{\left( 3\right) }-\mathfrak{m}\left(
\mathbf{\tilde{b}}_{i}^{\left( 2\right) }\right) \gamma _{QP},
\end{equation*}%
where $\gamma _{QP}$ is the part of the curve $\gamma _{P}$ from $Q$ to $P$.
Similar arguments as in the previous case show that
\begin{eqnarray*}
&&\mathbf{M}_{\alpha }\left( G\right) -\mathbf{M}_{\alpha }\left( \bar{G}%
\right) \\
&\geq &d_{\alpha }\left( \mathfrak{m}\left( \mathbf{\tilde{b}}_{i}^{\left(
2\right) }\right) \delta _{P},\mathbf{\tilde{b}}_{i}^{\left( 2\right)
}\right) -d_{\alpha }\left( \mathfrak{m}\left( \mathbf{\tilde{b}}%
_{i}^{\left( 2\right) }\right) \delta _{Q},\mathbf{\tilde{b}}_{i}^{\left(
2\right) }\right) \\
&\geq &\mathfrak{m}\left( \mathbf{\tilde{b}}_{i}^{\left( 2\right) }\right)
^{\alpha }\left( t_{2}-\pi \left( P\right) -CR_{i}\right) >0\text{, by Lemma
}\ref{Comparison}\text{ }.
\end{eqnarray*}%
Thus $\mathbf{M}_{\alpha }\left( G\right) >\mathbf{M}_{\alpha }\left( \bar{G}%
\right) $, which contradicts the optimality of $G$. Therefore, the
inequality (\ref{projection_closeness}) must hold.
\end{proof}

The following corollary states a scenario when a factory is located far away
from the community of households, a planner will never assign any production
to this factory under any optimal assignment map.

\begin{corollary}
\label{far away}Suppose for some $i\in \left\{ 1,\cdots ,k\right\} $,
\begin{equation}
|\pi \left( x_{i}\right) -\pi \left( y_{j}\right) |>2CR+\min_{i^{\ast }\neq
i}|\pi \left( x_{i^{\ast }}\right) -\pi \left( y_{j}\right) |
\label{projection_far}
\end{equation}%
for each $j=1,\cdots ,\ell $, where $C$ is the constant given in (\ref%
{ConstantC}) and
\begin{equation}
R=\max \left\{ \left\Vert z-p-\pi \left( z\right) v\right\Vert :z\in \left\{
y_{1},\cdots ,y_{\ell }\right\} \cup \left\{ x_{1},\cdots ,x_{k}\right\}
\right\} .  \label{constant_R}
\end{equation}%
Then, $S^{-1}\left( i\right) =\emptyset $ for any optimal assignment map $%
S\in Map\left[ \ell ,k\right] $.
\end{corollary}

\begin{proof}
Assume there exists $j\in \left\{ 1,\cdots ,\ell \right\} $ with $S\left(
j\right) =i$. Without loss of generality, we may assume $\pi \left(
x_{i}\right) \geq \pi \left( y_{j}\right) $. Thus, there exists an $j^{\ast
}\in S^{-1}\left( i\right) $ such that
\begin{equation*}
\pi \left( y_{j^{\ast }}\right) =\max \left\{ \pi \left( y_{h}\right) :\pi
\left( x_{i}\right) \geq \pi \left( y_{h}\right) \text{, }S\left( h\right)
=i\right\} .
\end{equation*}%
For this $j^{\ast }$, by Theorem \ref{autarky}, there exists a $z\in \Psi
_{i}$ with $\pi \left( y_{j^{\ast }}\right) <\pi \left( z\right) \leq \pi
\left( x_{i}\right) $ satisfying (\ref{projection_closeness}). By the
maximality of $\pi \left( y_{j^{\ast }}\right) $, $z\neq y_{h}$ for any $%
y_{h}\in \Psi _{i}$. Thus, $z=x_{i}$ and by (\ref{projection_far}),
\begin{eqnarray*}
\pi \left( x_{i}\right) &>&\pi \left( y_{j^{\ast }}\right)
+2CR+\min_{i^{\ast }\neq i}|\pi \left( x_{i^{\ast }}\right) -\pi \left(
y_{j^{\ast }}\right) | \\
&\geq &\pi \left( y_{j^{\ast }}\right) +2CR_{i}+\min_{i^{\ast }\neq i}|\pi
\left( x_{i^{\ast }}\right) -\pi \left( y_{j^{\ast }}\right) |,
\end{eqnarray*}%
a contradiction with (\ref{projection_closeness}).
\end{proof}

The next corollary shows an \textquotedblleft autarky\textquotedblright\
situation: if households and factories are located on two disjoint areas
lying distant away from each other, then the demand of households will
solely be satisfied from factories within the same area.

\begin{corollary}
\label{autarky1}Suppose
\begin{equation}
\left\{ \pi \left( x_{1}\right) ,\cdots ,\pi \left( x_{k}\right) ,\pi \left(
y_{1}\right) ,\cdots ,\pi \left( y_{\ell }\right) \right\} \cap \left(
t_{1},t_{2}\right) =\emptyset  \label{empty_zone}
\end{equation}%
for some $t_{1},$ $t_{2}\in \mathbb{R}$ with
\begin{equation*}
t_{2}=t_{1}+2CR+\sigma ,\sigma >0
\end{equation*}%
where the constants $C$ and $R$ are given in (\ref{ConstantC}) and (\ref%
{constant_R}). If
\begin{equation}
\left\{ \pi \left( x_{1}\right) ,\cdots ,\pi \left( x_{k}\right) \right\}
\cap (t_{1}-\sigma ,t_{1}]\neq \emptyset \text{ and }\left\{ \pi \left(
x_{1}\right) ,\cdots ,\pi \left( x_{k}\right) \right\} \cap \lbrack
t_{2},t_{2}+\sigma )\neq \emptyset ,  \label{nonempty_zone}
\end{equation}%
then for any optimal assignment map $S\in Map\left[ \ell ,k\right] $ and
interval $I=(-\infty ,t_{1}]$ or $[t_{2},\infty )$, we have
\begin{equation*}
\pi \left( y_{j}\right) \in I\iff \pi \left( x_{S\left( j\right) }\right)
\in I,
\end{equation*}%
for $j=1,\cdots ,\ell $.
\end{corollary}

\begin{proof}
It is sufficient to prove that if $\pi \left( x_{i}\right) \in \lbrack
t_{2},\infty )$ for some $i$, then the set
\begin{equation*}
\left\{ \pi \left( y_{h}\right) :S\left( h\right) =i\text{, }\pi \left(
y_{h}\right) \leq t_{1}\right\}
\end{equation*}%
must be empty. Indeed, if not, pick
\begin{equation*}
\pi \left( y_{j}\right) =\max \left\{ \pi \left( y_{h}\right) :S\left(
h\right) =i\text{, }\pi \left( y_{h}\right) \leq t_{1}\right\}
\end{equation*}%
for some $j$. By Theorem \ref{autarky}, there exists $z\in \Psi _{i}$ such
that $\pi \left( y_{j}\right) <\pi \left( z\right) \leq \pi \left(
x_{i}\right) $, and
\begin{eqnarray*}
\pi \left( z\right) -\pi \left( y_{j}\right) &\leq &2CR_{i}+\min_{i^{\ast
}\neq i}\left\vert \pi \left( x_{i^{\ast }}\right) -\pi \left( y_{j}\right)
\right\vert \\
&<&2CR_{i}+t_{1}-\pi \left( y_{j}\right) +\sigma \text{, by (\ref%
{nonempty_zone})} \\
&\leq &t_{2}-\pi \left( y_{j}\right) .
\end{eqnarray*}%
Thus, $\pi \left( z\right) <t_{2}$. On the other hand, the maximality of $%
\pi \left( y_{j}\right) $ and (\ref{empty_zone}) yield $\pi \left( z\right)
\geq t_{2}$, a contradiction.
\end{proof}

As a direct application of Corollary \ref{autarky1}, the next corollary
states that households living in a relatively isolated area are more likely
to receive their commodity from local factories.

\begin{corollary}
\label{autarky2}Let $t_{1}^{-}<$ $t_{2}^{-}<t_{1}^{+}<t_{2}^{+}$ be real
numbers with
\begin{equation*}
t_{2}^{-}=t_{1}^{-}+2CR+\sigma \text{ and }t_{1}^{+}=t_{1}^{+}+2CR+\sigma ,%
\text{ }\sigma >0,
\end{equation*}%
where the\ constants $C$ and $R$ are given in (\ref{ConstantC}) and (\ref%
{constant_R}). If
\begin{equation*}
\left\{ \pi \left( x_{1}\right) ,\cdots ,\pi \left( x_{k}\right) ,\pi \left(
y_{1}\right) ,\cdots ,\pi \left( y_{\ell }\right) \right\} \cap \left(
\left( t_{1}^{-},t_{2}^{-}\right) \cup \left( t_{1}^{+},t_{2}^{+}\right)
\right) =\emptyset ,
\end{equation*}%
\begin{equation}
\left\{ \pi \left( x_{1}\right) ,\cdots ,\pi \left( x_{k}\right) \right\}
\cap \left[ t_{2}^{-},t_{1}^{+}\right] =\left\{ \pi \left( x_{i}\right)
\right\}  \label{single_factory}
\end{equation}%
for some $i\in \left\{ 1,\cdots ,k\right\} $, and
\begin{equation*}
\left\{ \pi \left( x_{1}\right) ,\cdots ,\pi \left( x_{k}\right) \right\}
\cap (t_{1}^{-}-\sigma ,t_{1}^{-}]\neq \emptyset \text{, }\left\{ \pi \left(
x_{1}\right) ,\cdots ,\pi \left( x_{k}\right) \right\} \cap \lbrack
t_{2}^{+},t_{2}^{+}+\sigma )\neq \emptyset ,
\end{equation*}%
then for any optimal assignment map $S\in Map\left[ \ell ,k\right] $,
\begin{equation*}
S^{-1}\left( i\right) =\left\{ j:\pi \left( y_{j}\right) \in \left[
t_{2}^{-},t_{1}^{+}\right] \right\} .
\end{equation*}
\end{corollary}

\begin{proof}
For any $j\in S^{-1}\left( i\right) $, using $t_{1}=t_{1}^{-}$, $%
t_{2}=t_{2}^{-}$ in Corollary \ref{autarky1}, and the fact $\pi \left(
x_{i}\right) \in \lbrack t_{2}^{-},\infty )$, we have $\pi \left(
y_{j}\right) \in \lbrack t_{2}^{-},\infty )$. Similarly, using $%
t_{1}=t_{1}^{+}$, $t_{2}=t_{2}^{+}$ in Corollary \ref{autarky1}, and the
fact $\pi \left( x_{i}\right) \in (-\infty ,t_{1}^{+}]$, we have $\pi \left(
y_{j}\right) \in (-\infty ,t_{1}^{+}]$. Thus, $\pi \left( y_{j}\right) \in %
\left[ t_{2}^{-},t_{1}^{+}\right] $. \ This shows that $S^{-1}\left(
i\right) \subseteq \left\{ j:\pi \left( y_{j}\right) \in \left[
t_{2}^{-},t_{1}^{+}\right] \right\} $.

On the other hand, for any $j$ with $\pi \left( y_{j}\right) \in \left[
t_{2}^{-},t_{1}^{+}\right] $, we have $\pi \left( y_{j}\right) \in \lbrack
t_{2}^{-},\infty )$ and $\pi \left( y_{j}\right) \in (-\infty ,t_{1}^{+}]$.
Using Corollary \ref{autarky1} again, we have $\pi \left( x_{S\left(
j\right) }\right) \in \lbrack t_{2}^{-},\infty )$ and $\pi \left( x_{S\left(
j\right) }\right) \in (-\infty ,t_{1}^{+}]$. Thus, $\pi \left( x_{S\left(
j\right) }\right) \in \left[ t_{2}^{-},t_{1}^{+}\right] $. By (\ref%
{single_factory}), $S\left( j\right) =i$. Therefore, $\left\{ j:\pi \left(
y_{j}\right) \in \left[ t_{2}^{-},t_{1}^{+}\right] \right\} \subseteq
S^{-1}\left( i\right) $.
\end{proof}

\section{State matrix}

In this section, we show that the properties of optimal assignment maps
explored in previous sections can shed light on the search for those maps.
The analysis is built upon a notion of state matrix defined as follows.

\begin{definition}
Let $U=\left( u_{sh}\right) $ be an $k\times \ell $ matrix with $u_{sh}\in
\left\{ 0,1\right\} $. The matrix $U$ is called

\begin{enumerate}
\item a state matrix for an optimal assignment map $S$ if $S\left( h\right)
\neq s$ whenever $u_{sh}=0$.

\item a uniform state matrix if $U$ is a state matrix for any optimal
assignment map.
\end{enumerate}
\end{definition}

One could think of a state matrix as an information set of a planner during
the search process for optimal assignment maps. An entry $u_{sh}=0$ (or $%
u_{sh}=1$) simply denotes that the planner has (or has not) excluded the
possibility of assigning household $h$ to factory $s$. Recall that finding
an optimal assignment map is to minimize the functional $\mathbf{E}_{\alpha
}\left( S;\mathbf{x},\mathbf{b}\right) $ over the set $Map\left[ \ell ,k%
\right] $ whose cardinality is $k^{\ell }$. Any zero entry of a state matrix
$U$ for an optimal assignment map $S$ may exclude as many as $k^{\ell -1}$
assignment maps in $Map\left[ \ell ,k\right] $ from being $S$. The more zero
entries in a state matrix $U$, the more information about $S$ is contained
in $U$. Consequently, we aim at finding a state matrix $U$~for $S$ with as
many zero entries as possible, using properties of optimal assignment maps
studied in previous sections. When $U$ has exactly one non-zero entry in
each column, $S$ is completely determined by those non-zero entries in $U$.

We first explore the implication of Theorem \ref{Theorem marginal} on the
search for optimal assignment maps in the context of state matrix. For any
state matrix $U=\left( u_{ij}\right) $, we consider a $k\times \ell $ matrix
\begin{equation*}
W_{U}=\left( w_{ij}\left( U\right) \right) \text{ }
\end{equation*}%
where
\begin{equation*}
w_{ij}\left( U\right) =\rho _{\alpha }\left( w_{i}\left( U\right)
,n_{j}\right) \text{ with }w_{i}\left( U\right) :=\sum_{h=1}^{\ell
}u_{ih}n_{h},
\end{equation*}%
and the function $\rho _{\alpha }$ is given in (\ref{W-function}). Here, $%
w_{i}\left( U\right) $ denotes the maximum amount of commodity produced at
factory $i$ one could conjecture using the existing information in state
matrix $U.$

For any state matrix $U$, define%
\begin{equation*}
\digamma ^{s}\left( U;n_{j}\right) :=\bigcup_{i\neq s}\left\{ z\in \mathbb{R}%
^{m}:\left\Vert z-x_{i}\right\Vert <w_{sj}\left( U\right) \left\Vert
z-x_{s}\right\Vert \right\}
\end{equation*}%
and%
\begin{equation*}
\Omega ^{s}\left( U;n_{j}\right) :=\bigcap_{i\neq s}\left\{ z\in \mathbb{R}%
^{m}:\left\Vert z-x_{s}\right\Vert <w_{ij}\left( U\right) \left\Vert
z-x_{i}\right\Vert \right\} ,
\end{equation*}%
for any $s=1,\cdots ,k$ and $j=1,\cdots ,\ell $. By Lemma \ref{lemma ball},
each $\digamma ^{s}\left( U;n_{j}\right) $ is the union of $k-1$ open balls%
\begin{equation*}
\digamma ^{s}\left( U;n_{j}\right) =\bigcup_{i\neq s}B\left( x_{i}+\frac{%
\left( w_{sj}\left( U\right) \right) ^{2}}{1-\left( w_{sj}\left( U\right)
\right) ^{2}}\left( x_{i}-x_{s}\right) ,\frac{w_{sj}\left( U\right) }{%
1-\left( w_{sj}\left( U\right) \right) ^{2}}\left\Vert
x_{s}-x_{i}\right\Vert \right) ,
\end{equation*}%
and each set $\Omega ^{s}\left( U;n_{j}\right) $ is the intersection of $k-1$
open balls%
\begin{equation*}
\Omega ^{s}\left( U;n_{j}\right) =\bigcap_{i\neq s}B\left( x_{s}+\frac{%
\left( w_{ij}\left( U\right) \right) ^{2}}{1-\left( w_{ij}\left( U\right)
\right) ^{2}}\left( x_{s}-x_{i}\right) ,\frac{w_{ij}\left( U\right) }{%
1-\left( w_{ij}\left( U\right) \right) ^{2}}\left\Vert
x_{i}-x_{s}\right\Vert \right) .
\end{equation*}

\begin{example}
The matrix%
\begin{equation}
U^{\left( 0\right) }=\left( u_{ij}\right) \text{ with }u_{ij}=1\text{ for
any }i\text{ and }j  \label{U_0}
\end{equation}%
is a uniform state matrix. Then, $W_{U^{\left( 0\right) }}=\left(
w_{ij}\left( U^{\left( 0\right) }\right) \right) $ with
\begin{equation*}
w_{ij}\left( U^{\left( 0\right) }\right) =\rho _{\alpha }\left(
\sum_{h=1}^{\ell }n_{h},n_{j}\right) =\rho _{\alpha }\left( 1,n_{j}\right)
\end{equation*}%
which is independent of $i$. Here,
\begin{equation*}
\digamma ^{s}\left( U^{\left( 0\right) };n_{j}\right) =\digamma ^{s}\left(
n_{j}\right) \text{ and }\Omega ^{s}\left( U^{\left( 0\right) };n_{j}\right)
=\Omega ^{s}\left( n_{j}\right) ,
\end{equation*}%
where $\digamma ^{s}\left( n_{j}\right) $ and $\Omega ^{s}\left(
n_{j}\right) $ are given in (\ref{F0(nj)}) and (\ref{Omega0(nj)}).
\end{example}

\begin{example}
Let $S\in Map\left[ \ell ,k\right] $ be an optimal assignment map. Define
\begin{equation}
U_{S}=\left( u_{ij}\right)  \label{U_S}
\end{equation}%
with
\begin{equation*}
u_{ij}=\left\{
\begin{array}{cc}
1, & \text{if }S\left( j\right) =i \\
0, & \text{else}%
\end{array}%
\right. .
\end{equation*}%
Then, $W_{U_{S}}=\left( w_{ij}\left( U_{S}\right) \right) $ with
\begin{equation*}
w_{ij}\left( U_{S}\right) =\rho _{\alpha }\left( \sum_{S\left( h\right)
=i}n_{h},n_{j}\right) =\rho _{\alpha }\left( \mathfrak{m}\left( \mathbf{b}%
_{i}\right) ,n_{j}\right) .
\end{equation*}%
\ Note that
\begin{equation*}
\digamma ^{s}\left( U_{S};n_{j}\right) =\digamma _{S}^{s}\left( n_{j}\right)
\text{ and }\Omega ^{s}\left( U_{S};n_{j}\right) =\Omega _{S}^{s}\left(
n_{j}\right)
\end{equation*}%
where $\digamma _{S}^{s}\left( n_{j}\right) $ and $\Omega _{S}^{s}\left(
n_{j}\right) $ are given in (\ref{min_far}) and (\ref{min_close}).
\end{example}

\begin{definition}
Given two $k\times \ell $ real matrices $U=\left( u_{ij}\right) $ and $%
\tilde{U}=\left( \tilde{u}_{ij}\right) $, we define

\begin{enumerate}
\item $U\geq \tilde{U}$ if $u_{ij}\geq \tilde{u}_{ij}$ for each $i$ and $j$.

\item $U\gneq \tilde{U}$ if $U\geq \tilde{U}$ but $U\neq \tilde{U}$.
\end{enumerate}
\end{definition}

\begin{proposition}
Let $U$and $\tilde{U}$ be two state matrices for an optimal assignment map $%
S $. If $U\geq \tilde{U}$, then%
\begin{equation}
W_{U}\leq W_{\tilde{U}}  \label{W_comparison}
\end{equation}%
and
\begin{equation}
x_{s}\in \Omega ^{s}\left( U;n_{j}\right) \subseteq \Omega ^{s}\left( \tilde{%
U};n_{j}\right) \text{ and }\digamma ^{s}\left( U;n_{j}\right) \subseteq
\digamma ^{s}\left( \tilde{U};n_{j}\right) .  \label{Omega_inclusion}
\end{equation}
\end{proposition}

\begin{proof}
For each $i$, since $U\geq \tilde{U}$,
\begin{equation*}
w_{i}\left( U\right) =\sum_{h=1}^{\ell }u_{ih}n_{h}\geq \sum_{h=1}^{\ell }%
\tilde{u}_{ih}n_{h}=w_{i}\left( \tilde{U}\right) .
\end{equation*}%
Then, since $\rho _{\alpha }\left( \cdot ,n_{j}\right) $ is decreasing, it
follows
\begin{equation*}
w_{ij}\left( U\right) =\rho _{\alpha }\left( w_{i}\left( U\right)
,n_{j}\right) \leq \rho _{\alpha }\left( w_{i}\left( \tilde{U}\right)
,n_{j}\right) =w_{ij}\left( \tilde{U}\right)
\end{equation*}%
for each $i$ and $j$. By definition, we have both (\ref{W_comparison}) and (%
\ref{Omega_inclusion}).
\end{proof}

Let $U$ be a state matrix for an optimal assignment map $S$. By definitions
of $U^{\left( 0\right) }$ in (\ref{U_0}) and $U_{S}$ in (\ref{U_S}), it
follows that
\begin{equation}
U^{\left( 0\right) }\geq U\geq U_{S}.  \label{U_inclusion}
\end{equation}%
Thus, by (\ref{Omega_inclusion}), we have
\begin{eqnarray*}
\digamma ^{s}\left( U^{\left( 0\right) };n_{j}\right) &\subseteq &\digamma
^{s}\left( U;n_{j}\right) \subseteq \digamma ^{s}\left( U_{S};n_{j}\right) ,
\\
\Omega ^{s}\left( U^{\left( 0\right) };n_{j}\right) &\subseteq &\Omega
^{s}\left( U;n_{j}\right) \subseteq \Omega ^{s}\left( U_{S};n_{j}\right) .
\end{eqnarray*}%
These relations, together with Theorem \ref{Theorem marginal}, immediately
imply the following proposition:

\begin{proposition}
\label{U_1_Prop}Let $U=\left( u_{sj}\right) $ be a state matrix for an
optimal assignment map $S\in Map\left[ \ell ,k\right] $. For some $s$ and $j$%
,

\begin{enumerate}
\item if $y_{j}\in \digamma ^{s}\left( U;n_{j}\right) $, then $S\left(
j\right) \neq s$;

\item if $y_{j}\in \Omega ^{s}\left( U;n_{j}\right) $, then $S\left(
j\right) =s.$
\end{enumerate}
\end{proposition}

\begin{corollary}
Suppose $U$ is a state matrix for an optimal assignment map $S\in Map\left[
\ell ,k\right] $. Let $\widehat{U}^{\left( 1\right) }=\left( \widehat{u}%
_{ij}^{\left( 1\right) }\right) $ be a $k\times \ell $ matrix with
\begin{equation}
\widehat{u}_{ij}^{\left( 1\right) }=\left\{
\begin{array}{cc}
0, & \text{if }y_{j}\in \digamma ^{i}\left( U;n_{j}\right)  \\
u_{ij}, & \text{else}%
\end{array}%
\right. ,  \label{u_ij_1}
\end{equation}%
Then, $\widehat{U}^{\left( 1\right) }$ is also a state matrix for $S$ with $%
U\geq \widehat{U}^{\left( 1\right) }$.
\end{corollary}

\begin{proof}
If $\widehat{u}_{ij}^{\left( 1\right) }=0$, then either $u_{ij}=0$ or $%
y_{j}\in \digamma ^{i}\left( U;n_{j}\right) $. In the first case, since $U$
is a state matrix for $S$, by definition, $S\left( j\right) \neq i$. In the
second case, by Proposition \ref{U_1_Prop}, $S\left( j\right) \neq i$. Thus,
$\widehat{U}^{\left( 1\right) }$ is also a state matrix for $S$ with $U\geq
\widehat{U}^{\left( 1\right) }$.
\end{proof}

We now explore the implication of Theorem \ref{Theorem_neighborhood} on the
search for optimal assignment maps. Suppose $U$ is a state matrix for an
optimal assignment map $S$. If $u_{sh}=0$ for some $h\in \left\{ 1,\cdots
,\ell \right\} $ and $s\in \left\{ 1,\cdots ,k\right\} $, then for each $%
j\neq h$ with $n_{j}\leq n_{h}$, we consider the set%
\begin{equation}
\digamma ^{s,h}\left( U;n_{j}\right) :=\left\{ z\in \mathbb{R}%
^{m}:\left\Vert z-y_{h}\right\Vert +\Lambda \left( U\right) <w_{sj}\left(
U\right) \left\Vert z-x_{s}\right\Vert \right\} ,  \label{FishU}
\end{equation}%
where
\begin{equation*}
\Lambda \left( U\right) =\max \left\{ \frac{\rho _{\alpha }\left(
n_{h}+n_{j},n_{j}\right) }{w_{ij}\left( U\right) }\left\Vert
y_{h}-x_{i}\right\Vert :\text{for }i\in \left\{ 1,\cdots ,k\right\} \text{
with }u_{ih}=1\right\} .
\end{equation*}

\begin{lemma}
Let $U$ and $\tilde{U}$ be two state matrices for an optimal assignment map $%
S$. If $U\geq \tilde{U}$, then
\begin{equation}
\digamma ^{s,h}\left( U;n_{j}\right) \subseteq \digamma ^{s,h}\left( \tilde{U%
};n_{j}\right)  \label{FishU_comparison}
\end{equation}%
for any $s\in \left\{ 1,\cdots ,k\right\} $, $h\in \left\{ 1,\cdots ,\ell
\right\} $ with $u_{sh}=0$, and $n_{j}\leq n_{h}$ for $j\neq h$.
\end{lemma}

\begin{proof}
For each $i$, if $\tilde{u}_{ih}=1$, then $u_{ih}=1$ as $U\geq \tilde{U}$.
By (\ref{W_comparison}), we have $w_{ij}\left( U\right) \leq w_{ij}\left(
\tilde{U}\right) $. Thus,
\begin{eqnarray*}
\Lambda \left( \tilde{U}\right) &=&\max \left\{ \frac{\rho _{\alpha }\left(
n_{h}+n_{j},n_{j}\right) }{w_{ij}\left( \tilde{U}\right) }\left\Vert
y_{h}-x_{i}\right\Vert :\text{for }i\text{ with }\tilde{u}_{ih}=1\right\} \\
&\leq &\max \left\{ \frac{\rho _{\alpha }\left( n_{h}+n_{j},n_{j}\right) }{%
w_{ij}\left( U\right) }\left\Vert y_{h}-x_{i}\right\Vert :\text{for }i\text{
with }\tilde{u}_{ih}=1\right\} \\
&\leq &\max \left\{ \frac{\rho _{\alpha }\left( n_{h}+n_{j},n_{j}\right) }{%
w_{ij}\left( U\right) }\left\Vert y_{h}-x_{i}\right\Vert :\text{for }i\text{
with }u_{ih}=1\right\} =\Lambda \left( U\right) .
\end{eqnarray*}%
Consequently, (\ref{FishU_comparison}) follows from (\ref{FishU}).
\end{proof}

As a result, by (\ref{U_inclusion}),\
\begin{equation*}
\digamma ^{s,h}\left( U_{0};n_{j}\right) \subseteq \digamma ^{s,h}\left(
U;n_{j}\right) \subseteq \digamma ^{s,h}\left( U_{S};n_{j}\right) =\digamma
_{S}^{s,h}\left( n_{j}\right) ,
\end{equation*}%
where $\digamma _{S}^{s,h}\left( n_{j}\right) $ is given in (\ref%
{min_far_neigh}). The following proposition and its associated corollary
follow from Theorem \ref{Theorem_neighborhood}.

\begin{proposition}
Suppose $U$ is a state matrix for an optimal assignment map $S\in Map\left[
\ell ,k\right] $. If $y_{j}\in \digamma ^{s,h}\left( U;n_{j}\right) $ for
some $h\neq j$ with $u_{sh}=0$ and $n_{j}\leq n_{h}$, then $S\left( j\right)
\neq s$.
\end{proposition}

\begin{corollary}
Suppose $U$ is a state matrix for an optimal assignment map $S\in Map\left[
\ell ,k\right] $. Let $\widehat{U}^{\left( 2\right) }=\left( \widehat{u}%
_{sj}^{\left( 2\right) }\right) $ be a $k\times \ell $ matrix with
\begin{equation}
\widehat{u}_{sj}^{\left( 2\right) }=\left\{
\begin{array}{cc}
0, & \text{if }y_{j}\in \digamma ^{s,h}\left( U;n_{j}\right) \text{ for some
}h\neq j\text{ with }u_{sh}=0\text{ and }n_{j}\leq n_{h} \\
u_{sj}, & \text{else}%
\end{array}%
\right. ,  \label{nbhd_U}
\end{equation}%
Then, $\widehat{U}^{\left( 2\right) }$ is also a state matrix for $S$ with $%
U\geq \widehat{U}^{\left( 2\right) }$.
\end{corollary}

We now explore the implication of Theorem \ref{autarky} on the search for
optimal assignment maps. Suppose $U$ is a state matrix for an optimal
assignment map $S$. For each $i\in \left\{ 1,\cdots ,k\right\} $, let
\begin{equation*}
\Psi _{i}\left( U\right) =\left\{ y_{j}:u_{ij}=1\right\} \cup \left\{
x_{i}\right\} .
\end{equation*}%
Clearly,
\begin{equation*}
\Psi _{i}\left( U\right) \supseteq \Psi _{i}\left( U_{S}\right) =\Psi _{i},
\end{equation*}%
where $\Psi _{i}$ is defined in (\ref{Psi}).

Now, for $\pi :\mathbb{R}^{m}\rightarrow \mathbb{R}$ given in (\ref%
{pi_projection}), we define
\begin{equation}
R_{i}=\max \left\{ \left\Vert z-p-\pi \left( z\right) v\right\Vert :z\in
\Psi _{i}\left( U\right) \right\} .  \label{R_i_U}
\end{equation}%
Without loss of generality, we may assume that
\begin{equation*}
\Psi _{i}\left( U\right) =\left\{ y_{j_{h}}:h=1,\cdots ,N_{i}\right\} \cup
\left\{ x_{i}\right\}
\end{equation*}%
with
\begin{equation}
\pi \left( y_{j_{1}}\right) \leq \pi \left( y_{j_{2}}\right) \leq \cdots
\leq \pi \left( y_{j_{N_{i}}}\right) .  \label{order_pi_i}
\end{equation}

\begin{proposition}
Suppose $U=\left( u_{sj}\right) $ is a state matrix for an optimal
assignment map $S\in Map\left[ \ell ,k\right] $. For each $i\in \left\{
1,\cdots ,k\right\} $, let $h\in \left\{ 1,\cdots ,N_{i}\right\} $ and $%
i^{\ast }\in \left\{ 1,\cdots ,k\right\} $. If
\begin{equation}
\min \left\{ \pi \left( y_{j_{h+1}}\right) -\pi \left( y_{j_{h}}\right) ,\pi
\left( x_{i}\right) -\pi \left( y_{j_{h}}\right) \right\} >2CR_{i}+|\pi
\left( x_{i^{\ast }}\right) -\pi \left( y_{j_{h}}\right) |,  \label{tunnel1}
\end{equation}%
where $C$ and $R_{i}$ are the constants given in (\ref{ConstantC}) and (\ref%
{R_i_U}) respectively, then $S\left( j_{t}\right) \neq i$ for any $t\leq h$.
Similarly, if
\begin{equation}
\min \left\{ \pi \left( y_{j_{h}}\right) -\pi \left( y_{j_{h-1}}\right) ,\pi
\left( y_{j_{h}}\right) -\pi \left( x_{i}\right) \right\} >2CR_{i}+|\pi
\left( x_{i^{\ast }}\right) -\pi \left( y_{j_{h}}\right) |,  \label{tunnel2}
\end{equation}%
then $S\left( j_{t}\right) \neq i$ for any $t\geq h$.
\end{proposition}

\begin{proof}
Assume (\ref{tunnel1}) holds but $S\left( j_{t^{\ast }}\right) =i$ for some $%
t^{\ast }\leq h$. Without loss of generality, we may assume that
\begin{equation*}
\pi \left( y_{j_{t^{\ast }}}\right) =\max \left\{ \pi \left(
y_{j_{t}}\right) :t\leq h,S\left( j_{t}\right) =i\right\} .
\end{equation*}%
Note that $\pi \left( y_{j_{t^{\ast }}}\right) \leq \pi \left(
y_{j_{h}}\right) <\pi \left( x_{i}\right) $ by (\ref{order_pi_i}) and (\ref%
{tunnel1}). Then, by Theorem \ref{autarky}, there exists a $z\in \Psi
_{i}\backslash \left\{ y_{j_{t^{\ast }}}\right\} $ with $\pi \left(
y_{j_{t^{\ast }}}\right) <\pi \left( z\right) \leq \pi \left( x_{i}\right) $
such that
\begin{equation}
0<\pi \left( z\right) -\pi \left( y_{j_{t^{\ast }}}\right) \leq
2CR_{i}+\left\vert \pi \left( x_{i^{\ast }}\right) -\pi \left( y_{j_{t^{\ast
}}}\right) \right\vert .  \label{turnel_t_star}
\end{equation}%
By the maximality of $\pi \left( y_{j_{t^{\ast }}}\right) $ and $z\in \Psi
_{i}\backslash \left\{ y_{j_{t^{\ast }}}\right\} $, we know $\pi \left(
y_{j_{h}}\right) <\pi \left( z\right) $. Thus, by the ordering in (\ref%
{order_pi_i}), we have
\begin{equation}
\min \left\{ \pi \left( y_{j_{h+1}}\right) ,\pi \left( x_{i}\right) \right\}
\leq \pi \left( z\right) \text{.}  \label{comparison_pi_z}
\end{equation}%
Therefore,
\begin{eqnarray*}
&&\pi \left( z\right) -\pi \left( y_{j_{t^{\ast }}}\right) \\
&=&\pi \left( z\right) -\pi \left( y_{j_{h}}\right) +\pi \left(
y_{j_{h}}\right) -\pi \left( y_{j_{t^{\ast }}}\right) \\
&\geq &\min \left\{ \pi \left( y_{j_{h+1}}\right) -\pi \left(
y_{j_{h}}\right) ,\pi \left( x_{i}\right) -\pi \left( y_{j_{h}}\right)
\right\} +\pi \left( y_{j_{h}}\right) -\pi \left( y_{j_{t^{\ast }}}\right)
\text{, by (\ref{comparison_pi_z}) } \\
&>&2CR_{i}+|\pi \left( x_{i^{\ast }}\right) -\pi \left( y_{j_{h}}\right)
|+\pi \left( y_{j_{h}}\right) -\pi \left( y_{j_{t^{\ast }}}\right) \text{,
by (\ref{tunnel1})} \\
&\geq &2CR_{i}+|\pi \left( x_{i^{\ast }}\right) -\pi \left( y_{j_{t^{\ast
}}}\right) |,
\end{eqnarray*}%
a contradiction with (\ref{turnel_t_star}). This proves (\ref{tunnel1}).
Similar arguments give (\ref{tunnel2}).
\end{proof}

For each $i=\left\{ 1,\cdots ,k\right\} $, and $\lambda \in \mathbb{R}$,
denote
\begin{equation*}
I_{i}\left( \lambda \right) :=\left\{
\begin{array}{ll}
(-\infty ,\lambda ], & \text{if }\lambda \leq \pi \left( x_{i}\right) \\
\lbrack \lambda ,\infty )\text{,} & \text{if }\lambda >\pi \left(
x_{i}\right)%
\end{array}%
\right. .
\end{equation*}%
Then, for each $\pi :\mathbb{R}^{m}\rightarrow \mathbb{R}$ given in (\ref%
{pi_projection}), we define
\begin{equation*}
\tilde{\digamma}_{\pi }^{i}\left( U;n_{j}\right) :=\left\{ z\in \mathbb{R}%
^{m}:\pi \left( z\right) \in I_{i}\left( \pi \left( y_{j_{h}}\right) \right)
\text{ for some }j_{h}\text{ satisfying }(\ref{tunnel1})\text{ or }(\ref%
{tunnel2})\right\} .
\end{equation*}

\begin{corollary}
Suppose $U$ is a state matrix for an optimal assignment map $S\in Map\left[
\ell ,k\right] $. Let $\widehat{U}^{\left( 3\right) }=\left( \widehat{u}%
_{ij}^{\left( 3\right) }\right) $ be a $k\times \ell $ matrix with
\begin{equation}
\widehat{u}_{ij}^{\left( 3\right) }=\left\{
\begin{array}{cc}
0, & \text{if }y_{j}\in \tilde{\digamma}_{\pi }^{i}\left( U;n_{j}\right)
\text{ for some }\pi  \\
u_{ij}, & \text{else}%
\end{array}%
\right. ,  \label{u_ij_3}
\end{equation}%
Then, $\widehat{U}^{\left( 3\right) }$ is also a state matrix for $S$ with $%
U\geq \widehat{U}^{\left( 3\right) }$.
\end{corollary}

\begin{remark}
Depending on spatial locations of households and factories, for each fixed $%
i\in \left\{ 1,\cdots ,k\right\} $, the planner may choose $\pi $ to be one
of the standard coordinate functions in $\mathbb{R}^{m}$, i.e. $\pi \left(
z_{1},\cdots ,z_{m}\right) =z_{t}$ for some fixed $1\leq t\leq m$. In this
case, (\ref{tunnel1}) and (\ref{tunnel2}) may be simply expressed in terms
of coordinates of $x_{i}$'s and $y_{j}$'s. Another reasonable choice is to
set $\pi \left( z\right) =\left\langle z-p_{i},v_{i}\right\rangle $, where
\begin{equation*}
\left( p_{i},v_{i}\right) \in \arg \min \left\{ \max_{z\in \Psi _{i}\left(
U\right) }\left\Vert z-p-\left\langle z-p,v\right\rangle v\right\Vert
:p,v\in \mathbb{R}^{m}\text{ with }\left\Vert v\right\Vert =1\right\} .
\end{equation*}%
This will minimize $R_{i}$ given in (\ref{R_i_U}), because the line passing
through $p_{i}$ in direction $v_{i},$ i.e.
\begin{equation*}
\left\{ p_{i}+tv_{i}:t\in \mathbb{R}\right\} ,
\end{equation*}%
provides the least supremum norm approximation for $\Psi _{i}\left( U\right)
$ in $\mathbb{R}^{m}$.
\end{remark}

Given a state matrix $U$ for an optimal assignment map $S$, we have used
results from previous sections to provide three updated state matrices $\hat{%
U}^{\left( j\right) }$, $j=1,2,3$, for $U$. The next proposition makes it
possible to combine them together into a further updated state matrix.

\begin{proposition}
\label{prop min}Suppose $U=\left( u_{ij}\right) $ and $\bar{U}=\left( \bar{u}%
_{ij}\right) $ are two state matrices for an optimal assignment map $S\in Map%
\left[ \ell ,k\right] $. Then, the matrix $\tilde{U}=\left( \tilde{u}%
_{ij}\right) $ given by%
\begin{equation*}
\tilde{u}_{ij}=\min \left\{ u_{ij},\bar{u}_{ij}\right\} \text{ for all }i%
\text{ and }j
\end{equation*}%
is also a state matrix for $S$.
\end{proposition}

\begin{proof}
If $\tilde{u}_{ij}=0$, then either $u_{ij}=0$ or $\tilde{u}_{ij}=0$. Both
cases give $S\left( j\right) \neq i$.
\end{proof}

Proposition \ref{prop min} says that one could deduce more information from
any two existing state matrices regarding the optimal assignment map. Using
this proposition, we immediately have the following corollary.

\begin{corollary}
\label{U updating}Suppose $U$ is a state matrix for an optimal assignment
map $S\in Map\left[ \ell ,k\right] $. For each $i$ and $j$, define
\begin{equation*}
\widehat{u}_{ij}=\min \left\{ \widehat{u}_{ij}^{\left( 1\right) },\widehat{u}%
_{ij}^{\left( 2\right) },\widehat{u}_{ij}^{\left( 3\right) }\right\} ,
\end{equation*}%
where $\widehat{u}_{ij}^{\left( 1\right) }$, $\widehat{u}_{ij}^{\left(
2\right) }$ and $\widehat{u}_{ij}^{\left( 3\right) }$ are given in (\ref%
{u_ij_1}), (\ref{nbhd_U}) and (\ref{u_ij_3}) respectively. Then, $\widehat{U}%
=\left( \widehat{u}_{ij}\right) $ is also a state matrix for $S$ with $U\geq
\widehat{U}$.
\end{corollary}

This idea of updating a state matrix $U$ into another state matrix $\widehat{%
U}$ as in Corollary \ref{U updating} can be implemented iteratively to
obtain an even further updated state matrix. Given any initial state matrix $%
U$ (e.g. $U=U^{\left( 0\right) }$ as in (\ref{U_0})) for an optimal
assignment map $S$. For each $n=0,1,2,\cdots ,$ define
\begin{equation*}
U_{n+1}=\widehat{U_{n}}\text{ with }U_{0}=U\text{.}
\end{equation*}%
This gives a non-increasing sequence of $k\times \ell $ matrices $\left\{
U_{n}\right\} $ whose entries are either $0$ or $1$. Hence, there exists an $%
N\geq 1$ such that
\begin{equation*}
U_{0}\gneq U_{1}\gneq \cdots \gneq U_{N-1}=U_{N}=U_{N+1}=\cdots .
\end{equation*}%
We denote this $U_{N}$ as $U^{\ast }$. Clearly, the matrix $U^{\ast }$ is
still a state matrix for $S$ with $U\geq U^{\ast }$ and $\widehat{U^{\ast }}%
=U^{\ast }$.

This updated state matrix $U^{\ast }$ contains more information about $S$
than the initial state matrix $U_{0}$ because $U^{\ast }$ contains more zero
entries. In some non-trivial cases as illustrated in the following example, $%
U^{\ast }$ may have exactly one non-zero entry in each column. In such a
situation, $U^{\ast }$ completely determines the optimal assignment map $S$.

\begin{example}
Let $U$ be a uniform state matrix (e.g. $U=U^{\left( 0\right) }$ as in (\ref%
{U_0})), and suppose that
\begin{equation*}
n_{1}\geq n_{2}\geq \cdots \geq n_{\ell }.
\end{equation*}%
If for each $j=1,\cdots ,\ell $,%
\begin{equation}
y_{j}\in \bigcap_{\substack{ u_{sj}=1  \\ s\neq s_{j}}}\left( \digamma
^{s}\left( U;n_{j}\right) \cup \bigcup_{\substack{ 1\leq h\leq j-1  \\ %
s_{h}\neq s}}\digamma ^{s,h}\left( U;n_{j}\right) \right)
\label{assumption_example}
\end{equation}%
for some $s_{j}\in \left\{ 1,\cdots ,k\right\} $, then $S:\left\{ 1,\cdots
,\ell \right\} \rightarrow \left\{ 1,\cdots ,k\right\} $ given by $S\left(
j\right) =s_{j}$ is the optimal assignment map.
\end{example}

\begin{proof}
It is sufficient to show that for any optimal assignment map $S$, it holds
that $S\left( j\right) =s_{j}$ for any $j$. Indeed, for any $s\neq s_{j}$,
if $u_{sj}=0$, then $S\left( j\right) \neq s$ because $U$ is a state matrix
for $S$. If $u_{sj}=1$, then by assumption (\ref{assumption_example}),
either $y_{j}\in \digamma ^{s}\left( U;n_{j}\right) $ or $\digamma
^{s,h}\left( U;n_{j}\right) $ for some $h<j$ with $s_{h}\neq s$. If $%
y_{j}\in \digamma ^{s}\left( U;n_{j}\right) $, then by Proposition \ref%
{U_1_Prop}, $S\left( j\right) \neq s$. If $y_{j}\in \digamma ^{s,h}\left(
U;n_{j}\right) $ for some $h<j$ with $s_{h}\neq s$, then either $S\left(
h\right) \neq s_{h}$ or $S\left( h\right) =s_{h}\neq s$. In the later case,
since $y_{j}\in \digamma ^{s,h}\left( U;n_{j}\right) \subseteq \digamma
_{S}^{s,h}\left( n_{j}\right) $ and $n_{j}\leq n_{h}$, by Theorem \ref%
{Theorem_neighborhood}, we still have $S\left( j\right) \neq s$. Thus, in
all cases for any $s\neq s_{j}$, we know
\begin{equation}
\text{either }S\left( j\right) \neq s\text{ or }S\left( h\right) \neq s_{h}%
\text{ for some }h<j\text{.}  \label{two_cases}
\end{equation}%
Consequently, when $j=1$, we always have $S\left( 1\right) \neq s$ for any $%
s\neq s_{1}$, and thus $S\left( 1\right) =s_{1}$. Using (\ref{two_cases})
again, we get $S\left( 2\right) \neq s$ for any $s\neq s_{2}$, which yields $%
S\left( 2\right) =s_{2}$. Repeating this process leads to the conclusion
that $S\left( j\right) =s_{j}\,$\ for any $j\in \left\{ 1,\cdots ,\ell
\right\} .$
\end{proof}

\section{Conclusion}

This paper proposes an optimal allocation problem with ramified transport
technology in a spatial economy. A planner needs to find an optimal
allocation plan as well as an associated optimal allocation path to minimize
overall cost of transporting commodity from factories to households. This
problem differentiates itself from existing ramified transportation
literature in that the distribution of production among factories is not
fixed but endogenously determined as in many allocation practices. It's
shown that due to the transport economy of scale in ramified transportation,
each optimal allocation plan corresponds equivalently to an optimal
assignment map from households to factories. This optimal assignment map
provides a natural partition of both households and allocation paths. We
develop methods of marginal transportation analysis and projectional
analysis to study properties of optimal assignment maps. These properties
are then related to the search for an optimal assignment map in the context
of state matrix.

The ramified optimal allocation problem studied in this paper provides a
prototype for a class of problems arising in spatial resource allocations.
One natural extension is to allow the locations of factories $\left\{
x_{1},x_{2},\cdots ,x_{k}\right\} $ to vary which then gives rise to an
optimal location problem. An analogous optimal location problem in
Monge-Kantorovich transportation has been extensively studied as in McAsey
and Mou \cite{mou1}, Morgan and Bolton \cite{morgan} and references therein.
Meanwhile, one may consider another extension of the ramified allocation
problem by generalizing the atomic measure $\mathbf{b}$ of households to an
arbitrary probability measure $\mu $, not necessarily atomic. In particular,
when $\mu $ represents the Lebesgue measure on a domain, a partition of $\mu
$ given by an optimal assignment map may analogously lead to a partition of
the domain. This consequently gives rise to an optimal partition problem of
dividing the given domain into $k$ regions according to ramified optimal
transportation.

%\bibitem{Ambrosio} Ambrosio, L.: Lecture notes on optimal transport
%problems. Mathematical aspects of evolving interfaces (Funchal, 2000),
%1--52, \textit{Lecture Notes in Math.}, 1812, Springer, Berlin, 2003.
%
%\bibitem{beckman} Beckman, M.: \textit{Lectures on location theory},
%Springer-Verlag, Berlin, 1999.

\end{document}